\documentclass[12pt,a4paper]{extarticle}

\usepackage{amsmath}
\usepackage{amssymb}
\usepackage{amsthm}
\usepackage{xcolor}
\usepackage{bbm}
\usepackage{graphicx}
\usepackage{verbatim}

\usepackage[T1]{fontenc}
\usepackage[ansinew]{inputenc}
\usepackage[german,english,french,USenglish]{babel}
\usepackage[english]{babel}

\makeatletter
\numberwithin{equation}{section}
\theoremstyle{plain}

\newtheorem{dummytheorem}{Dummy-Theorem}[section]
\newcommand{\proofendsign}{$\Box$} 

\newtheorem{lemma}[dummytheorem]{Lemma}
\newtheorem{theorem}[dummytheorem]{Theorem}
\newtheorem{proposition}[dummytheorem]{Proposition}

\newtheorem{example}[dummytheorem]{Example}

\newtheorem{remark}[dummytheorem]{Remark}

\newcommand{\N}{\mathbb{N}}

\newcommand{\Q}{\mathrm{Q}}
\newcommand{\R}{\mathbb{R}}

\newcommand{\E}{\mathbb{E}}

\newcommand{\pr}{\mathrm{P}}
\newcommand{\opr}{\overline{\mathrm{P}}}
\newcommand{\qr}{\mathrm{Q}}
\newcommand{\ex}{\mathbb{E}}

\newcommand{\eins}{\mathbbm{1}}

\newcommand{\cF}{\mathcal F}

\newcommand{\cL}{\mathcal L}
\newcommand{\cM}{\mathcal M}
\newcommand{\cP}{\mathcal P}
\newcommand{\cQ}{\mathcal Q}

\newcommand{\cS}{\mathcal S}
\newcommand{\cT}{\mathcal T}

\newcommand{\cX}{\mathcal X}

\newcommand{\hY}{\widehat{Y}}

\newcommand{\tY}{\widetilde{Y}}

\newcommand{\oY}{\overline{Y}}

\newcommand{\esssup}{\mathop{\mathrm{ess\,sup}}\displaylimits}
\newcommand{\essinf}{\mathop{\mathrm{ess\,inf}}\displaylimits}

\newcommand{\OFP}{(\Omega,{\cal F},\pr)}
\newcommand{\OFPT}{(\Omega,\cF,\pr)}

\newcommand{\tOFP}{(\overline{\Omega},\overline{{\cal F}},\overline{\pr})}

\newcommand{\OFFP}{(\Omega,{\cal F}, ({\cal F}_{t})_{0\leq t\leq T},\pr)}
\newcommand{\OFFQ}{(\Omega,{\cal F}_{T}, ({\cal F}_{t})_{0\leq t\leq T},\qr)}

\newcommand{\oOFFP}{(\overline{\Omega},\overline{{\cal F}}_{T},(\overline{{\cal F}}_{t})_{0\leq t\leq T},\overline{\pr})}
\newcommand{\oYk}{\overline{Y}^{k}}

\newcommand{\ocF}{\overline{{\mathcal F}}}

\newcommand{\taur}{\tau^{r}}

\usepackage[colorinlistoftodos,bordercolor=orange,backgroundcolor=orange!20,linecolor=orange,textsize=scriptsize]{todonotes}

\begin{document}
\title{Minimax theorems for American options in incomplete markets without time-consistency}
\author{Denis Belomestny\footnote{Faculty of Mathematics, University of Duisburg--Essen, {\tt denis.belomestny@uni-due.de}}
\and{\setcounter{footnote}{2}}Tobias H\"ubner\footnote{Faculty of Mathematics, University of Duisburg--Essen, {\tt tobias.huebner@stud.uni-due.de}}
\and{\setcounter{footnote}{3}}
Volker Krätschmer\footnote{Faculty of Mathematics, University of Duisburg--Essen, {\tt volker.kraetschmer@uni-due.de}}
\and{\setcounter{footnote}{4}}Sascha Nolte\footnote{Faculty of Mathematics, University of Duisburg--Essen, {\tt sascha.nolte@stud.uni-due.de}}
}

\date{}

\maketitle
\begin{abstract}
In this paper we give sufficient conditions guaranteeing the validity of the well-known minimax theorem for the lower Snell envelope with respect to a family of absolutely continuous  probability measures. Such minimax results play an important role in the characterisation of arbitrage-free prices of American contingent claims in incomplete markets. Our conditions do not rely on the notions of stability under pasting or time-consistency and reveal some unexpected connection between the minimax result and the path properties of the corresponding density process. 
\end{abstract}


\section{Introduction}
\label{setup}
Let \(0<T<\infty\) and let $\OFFP$  be a filtered probability space,
where $(\cF_{t})_{0 \leq t \leq T}$ is a right-continuous filtration with $\cF_{0}$ containing only the sets of probability $0$ or $1$ as well as all the null sets of $\cF_T.$ In the sequel we shall assume without loss of generality that 
$\cF = \cF_T$.  Furthermore, let $\cQ$ denote a nonvoid set of probability measures on $\cF$ 
all absolutely continuous w.r.t. \(\pr.\)
We shall denote by $\cL^{1}(\cQ)$ 
the set of all random variables $X$ on $(\Omega,\cF,\pr)$ which are $\qr$-integrable for every $\qr\in\cQ$. 
such that $\sup_{\qr\in\cQ}\ex_{\qr}[|X|] < \infty$. 
Let \(S=(S_{t})_{0 \leq t \leq T}\) be a semimartingale w.r.t. $(\cF_t)_{0 \leq t \leq T}$ whose trajectories are right continuous and have finite left limits (càdlàg).
Consider also  another right-continuous $(\mathcal{F}_t)$-adapted stochastic process $Y~ \dot=~(Y_{t})_{0 \leq t \leq T}$  with bounded paths,
and let $\cT$ stand for the set of all finite stopping times $\tau\leq T$ w.r.t. $(\mathcal{F}_t)_{0 \leq t \leq T}.$ We also assume that the process is  {\it quasi left-uppersemicontinuous} w.r.t. $\pr$, i.e. $\limsup_{n\to\infty}Y_{\tau_{n}}\leq Y_{\tau}$ $\pr$-a.s. holds for any sequence $(\tau_{n})_{n\in\N}$ in $\cT$ satisfying $\tau_{n}\nearrow\tau$ for some $\tau\in\cT$.  
The so-called {\it lower Snell envelope of $Y$} (w.r.t. \(\cQ\)) is the 
stochastic process $(U^{\downarrow, Y}_t)_{0 \leq t \leq T}$, defined via
$$
U^{\downarrow, Y}_{t}~\dot=~\essinf_{\qr\in\cQ}\esssup_{\tau\in\cT, \tau\geq t}\ex_{\qr}[Y_{\tau}|\cF_{t}] \quad (t\in [0,T]).
$$ 
The main  objective of our work is to find sufficient conditions for  the validity of the following minimax result 
\begin{equation}
\label{eq: minimax_intr}
\sup_{\tau\in\cT}\, \inf_{\qr\in\cQ}\,\ex_{\qr}[Y_{\tau}] = \inf_{\qr\in\cQ}\,\sup_{\tau\in\cT}\,\ex_{\qr}[Y_{\tau}] = U^{\downarrow, Y}_{0}.
\end{equation}
In financial mathematics this type of results appears to be useful in the characterisation of arbitrage-free prices of American contingent claims in incomplete markets. If \(\cM\) stands for the family of equivalent local  martingale measures w.r.t. $S$, i.e.
\begin{eqnarray*}
\cM=\{Q\sim P\, | \,  S \mbox{ is a local martingale under } Q\},
\end{eqnarray*}
then  the set  $\Pi(Y)$ of the so-called arbitrage-free prices for $Y$ with respect to \(\cM\) can be defined as the set of all real numbers $c$ fulfilling two  properties: 
\begin{enumerate}
\item[$(i)$]
 $c \leq \ex_\qr[Y_\tau]$ for some stopping time $\tau \in \cT$  and a martingale measure $\qr \in \cM$;
\item[$(ii)$]
for any stopping time $\tau' \in \cT$ there exists some $\qr' \in \cM$ such that $c \geq \ex_{\qr'}[Y_{\tau'}]$.
\end{enumerate}
The above definition implies that given \(c\in \Pi(Y),\) it holds
\begin{eqnarray*}
\sup_{\tau\in\cT} \inf_{\qr\in\cM}\ex_{\qr}[Y_\tau]\leq c \leq \sup_{\tau\in\cT} \sup_{\qr\in\cM}\ex_{\qr}[Y_\tau].
\end{eqnarray*}
In particular, the lower Snell-envelope at time $0$ gives the greatest lower bound for the arbitrage free price of an American option. The following important characterisation of the set  $\Pi(Y)$ can be found in \cite{Trevino2008} (Theorem 1.20) or \cite{Trevino2012}. 
\begin{theorem}
\label{thm: main_char}
Suppose that \(\{Y_{\tau}\mid \tau\in\cT\}\) is uniformly $\qr$-integrable 
for any \(\qr\in\cM\), and that $ Y = (Y_{t})_{0 \leq t \leq T}$ is upper-semicontinuous in expectation from the left w.r.t. every $\qr\in\cM$, i.e. $\limsup_{n\to\infty}\ex_{\qr}[Y_{\tau_n}]\leq\ex_{\qr}[Y_{\tau}]$ for any increasing sequence $(\tau_{n})_{n\in\N}$ converging to $\tau\in\cT$.  If \(\cM\) denotes the set of equivalent local martingale measures w.r.t. $S$, then the set $\Pi(Y)$ of arbitrage-free prices for $Y$ corresponding to $\cM$ is a real interval with endpoints
\begin{eqnarray*}
\pi_{\inf}(Y)\doteq\inf_{\qr\in\cM}\sup_{\tau\in\cT} \ex_{\qr}[Y_\tau]=\sup_{\tau\in\cT} \inf_{\qr\in\cM}\ex_{\qr}[Y_\tau]
\end{eqnarray*}
and
\begin{eqnarray*}
\pi_{\sup}(Y)\doteq\sup_{\qr\in\cM}\sup_{\tau\in\cT} \ex_{\qr}[Y_\tau]=\sup_{\tau\in\cT} \sup_{\qr\in\cM}\ex_{\qr}[Y_\tau].
\end{eqnarray*}
\end{theorem}
A natural question is whether a similar characterisation can be proved for sets \(\cQ\) of equivalent measures which are strictly ``smaller'' then \(\cM.\) This question can be interesting for at least two reasons. First, the buyer (or the seller) of the option may have some preferences about  the set of pricing measures \(\cQ\) resulting in some additional restrictions on \(\cQ\) such that \(\cQ\subset \cM\). Second, the set of all martingale measures \(\cM\) may be difficult to describe in a constructive way, as typically  only sufficient conditions   for the relation \(Q\in \cM\) are available.  
\par
A careful inspection of the proof of Theorem \ref{thm: main_char} reveals that it essentially relies on the minimax identity  \eqref{eq: minimax_intr}  with \(\cQ=\cM\) which is routinely proved in the literature using the so-called stability under pasting property of \(\cM\). To recall, a set ${\cQ}$ of probability measures on $\cF$ is called {\it stable under pasting w.r.t. $\OFFP$}, if each of them is equivalent to $\pr$, and for every $\qr_{1},\qr_{2}\in\cQ$ as well as $\tau\in\cT$ the {\it pasting of $\qr_{1}$ and $\qr_{2}$ in $\tau$}, i.e. the probability measure $\qr_{3}$ defined by the pasting procedure
$$
\qr_{3}(A)~ \dot=~\ex_{\qr_{1}}[\qr_{2}[A|\cF_{\tau}]]\quad(A\in\cF),
$$
belongs to $\cQ$.  The stability under pasting implies that the set \(\cQ\) is rather ``big'' and basically coincides with \(\cM\) if we exclude the trivial case where $\cQ$ consists of one element. 
Let us mention that the property of stability under pasting is closely related to the concept of time consistency. As in \cite{Delbaen2006} we shall call a set $\cQ$ of probability measures on $\cF$ to be {\it time-consistent w.r.t. $\OFFP$}, if for any $\tau,\sigma\in\cT$ with 
$\tau\leq\sigma$ and $\pr$-essentially bounded random variables $X,Z,$ we have the following implication
$$
\essinf_{\qr\in\cQ}\ex_{\qr}[X|\cF_{\sigma}]\leq\essinf_{\qr\in\cQ}\ex_{\qr}[Z|\cF_{\sigma}]\quad\Rightarrow\quad
\essinf_{\qr\in\cQ}\ex_{\qr}[X|\cF_{\tau}]\leq\essinf_{\qr\in\cQ}\ex_{\qr}[Z|\cF_{\tau}].
$$ 
Other contributions to the minimax-relationship \eqref{eq: minimax_intr} use the property of {\it recursiveness} (see \cite{BayraktarKaratzasYao2010},
\cite{BayraktarYao2011I}, \cite{BayraktarYao2011II}, \cite{ChengRiedel2013})
$$
\essinf_{\qr\in\cQ}\ex_{\qr}\left[\essinf_{\qr\in\cQ}\ex_{\qr}[X|\cF_{\sigma}]~\Big|~\cF_{\tau}\right] = \essinf_{\qr\in\cQ}\ex_{\qr}[X|\cF_{\tau}]
$$
holding for stopping times $\sigma,\tau\in\cT$ with $\tau \leq \sigma$ and any $\pr$-essentially bounded random variable $X$. It may be easily verified  that recursiveness and time-consistency are equivalent  (see e.g. proof of Theorem 12 in \cite{Delbaen2006}). Moreover,  stability under pasting generally implies time-consistency (see  Proposition \ref{bereits Zeitkonsistenz} below, and also \cite[Theorem 6.51]{FoellmerSchied2011} for the time-discrete case). To the best of our knowledge, all studies of the minimax-relationship  \eqref{eq: minimax_intr} so far considered  only time-consistent sets $\cQ$ (see Section~\ref{discussion} for a further discussion on this issue). 
\medskip

In this paper we formulate  conditions on the family \(\cQ\) of other sort which do not rely on the notions of consistency or stability but still ensure the minimax relation \eqref{eq: minimax_intr}.  
The key is to impose a certain condition on the range of the  mapping 
$$
\mu_{\cQ}: \cF\rightarrow l^{\infty}(\cQ),~\mu_{\cQ}(A)(\qr)~:=~\qr(A),
$$
where $l^{\infty}(\cQ)$ denotes the space of all bounded real-valued mappings on $\cQ$.
This is a so-called vector measure satisfying {$\mu_{\cQ}(A_1 \cup A_2) = \mu_{\cQ}(A_1)+\mu_{\cQ}(A_2)$ for disjoint sets $A_{1}, A_{2} \in\cF$. We shall refer to $\mu_{\cQ}$ as the {\it vector measure associated with $\cQ$}. 

The paper is organized as follows. In Section~\ref{main_results} we present our main result concerning the sets $\cQ$ whose associated vector measures have relatively compact range. 
Next we deduce another criterion in terms of path properties of the corresponding density process  $(d\qr/d\pr)_{\qr\in\cQ}$. The latter characterisation is especially  useful for the case of suitably parameterized families of local martingale measures. Specifically in Section~\ref{applications to parameterized families}  we formulate  an easy to check criterion  for the case of density processes corresponding to nearly sub-Gaussian families of local martingales. 
In Section~\ref{discussion} we shall discuss related  results from the literature. Section~\ref{general abstract minimax} contains a general minimax result for the lower Snell-envelopes. The proofs of all relevant results are gathered in Section~\ref{proofs}, whereas in the appendix some auxiliary results on path properties of nearly sub-Gaussian random fields are presented.

\section{Main results}
\label{main_results} 
Throughout this paper we assume that
\begin{equation}
\label{nonatomicdominated}
(\Omega,\cF_{t},\pr_{|\cF_{t}})~\mbox{is atomless for every}~t > 0.
\end{equation}
Concerning the process $Y,$ we shall assume that
\begin{equation}
\label{Integrierbarkeit}
Y^{*} := \sup_{t\in [0,T]}|Y_{t}|\in \cL^{1}(\cQ).
\end{equation}
Moreover, the space $l^{\infty}(\cQ)$
will be endowed with the sup-norm $\|\cdot\|_{\infty}$. 
Using the notation $co(\cQ)$ for the convex hull of $\cQ$, 
our main minimax result reads as follows.
\begin{theorem}
\label{secondMinimax}
Let the range of 
$\mu_{\cQ}$ 
be relatively $\|\cdot\|_{\infty}$-compact. If
$Y = (Y_{t})_{0 \leq t \leq T}$ fulfills \eqref{Integrierbarkeit} as well as  
$\sup_{\qr\in\cQ}\ex_{\qr}[Y^{*}\eins_{\{Y^{*} > a\}}]\to 0$ for $a\to\infty$, and if 
\eqref{nonatomicdominated} holds, 
then 
\begin{eqnarray}
\label{eq:minimax_rel}
\sup_{\tau\in\cT}\, \inf_{\qr\in \cQ}\,\ex_{\qr}[Y_{\tau}] 
= \sup_{\tau\in\cT}\, \inf_{\qr\in co(\cQ)}\,\ex_{\qr}[Y_{\tau}] 
= \inf_{\qr\in co(\cQ)}\,\sup_{\tau\in\cT}\,\ex_{\qr}[Y_{\tau}].
\end{eqnarray}
\end{theorem}
The proof of Theorem~\ref{secondMinimax} may be found in Section~\ref{proof of secondMinimax}.
\begin{remark}
\label{total variation}
Let $\cQ$ be relatively compact w.r.t. the topology of total variation, i.e. the topology with metric $d_{tv}$ defined by
$$
d_{tv}(\qr_{1},\qr_{2}) := \sup_{A\in\cF}|\qr_{1}(A) - \qr_{2}(A)|.
$$
Then it is already known that $\{\mu_{\cQ}(A)\mid A\in\cF\}$ is relatively $\|\cdot\|_{\infty}$-compact (cf. \cite{Amarante2014}). 
Moreover, if each member of $\cQ$ is equivalent to $\pr$, then the set $\cQ$ is \underline{not} time-consistent w.r.t. $\OFFP$ whenever it has more than one element,  $(\Omega,\cF_{t},\pr_{|\cF_{t}})$ is atomless and  
$L^{1}(\Omega,\cF_{t},\pr_{|\cF_{t}})$ is weakly separable for every $t > 0$. This will be shown in Section~\ref{compact-anti-zeitkonsistent}. 
The above conditions on the filtration 
$(\cF_t)_{0 \leq t \leq T}$ are always satisfied if it is assumed to be the standard augmentation of the natural filtration induced by some d-dimensional right-continuous stochastic process $Z = (Z_{t})_{0 \leq t \leq T}$ on the probability space $(\Omega,\cF,\pr)$ such that the marginals $Z_{t}$ have absolutely
continuous distributions for any $t > 0$, $Z_{0}$ is constant  $\pr$-a.s. and $\cF_0$ is trivial
(see \cite[Remark 2.3]{BelomestnyKraetschmer2017a}, or \cite[Remark 3]{BelomestnyKraetschmer2017b}).

\end{remark}

Let us now present a simple sufficient criterion guaranteeing the validity of the minimax relation \eqref{eq:minimax_rel}. It turns out that under these conditions $\cQ$ fails to be time-consistent.

\begin{theorem}
\label{point compact class}
Let 
the conditions \eqref{Integrierbarkeit} and \eqref{nonatomicdominated} be fulfilled and let $\sup_{\qr\in\cQ}\ex_{\qr}[Y^{*}\eins_{\{Y^{*} > a\}}]\to 0$ as $a\to\infty$. Furthermore, let $d$ denote a totally bounded semimetric on $\cQ$ and let $(d\qr/d\pr)_{\qr\in\cQ}$ have $\pr$-almost surely $d$-uniformly continuous paths. If $(d\qr/d\pr)_{\qr\in\cQ}$ is dominated by some $\pr$-integrable random variable, then  
$$
\sup_{\tau\in\cT}\, \inf_{\qr\in \cQ}\,\ex_{\qr}[Y_{\tau}] 
= \sup_{\tau\in\cT}\, \inf_{\qr\in co(\cQ)}\,\ex_{\qr}[Y_{\tau}] 
= \inf_{\qr\in co(\cQ)}\,\sup_{\tau\in\cT}\,\ex_{\qr}[Y_{\tau}].
$$
If in addition $L^{1}(\Omega,\cF_{t},\pr_{|\cF_{t}})$ is weakly separable for every $t > 0$, then $\cQ$ is not time-consistent w.r.t. $\OFFP$ whenever it consists of more than one element and all elements of \(\cQ\) are equivalent to $\pr$.
\end{theorem}

The proof of Theorem \ref{point compact class} is delegated to Section~\ref{Theorem point compact class}.
\section{Applications   to parameterized families}
\label{applications to parameterized families}
Fix  a semimetric space $(\Theta,d_{\Theta})$ with finite diameter $\Delta$. Moreover, let us assume 
\begin{equation}
\label{parametrization}
\cQ = \{\qr_{\theta}\mid \theta\in\Theta\}
\quad\mbox{and}\quad\qr_{\theta}\not=\qr_{\vartheta}~\mbox{for}~\theta\not=\vartheta.
\end{equation}
Then $d_{\Theta}$ induces in a natural way a semimetric $d$ on $\cQ$ which is totally bounded if and only if $d_{\Theta}$ fulfills this property. We want to find conditions such that $\cQ$ meets the requirements of Theorem \ref{point compact class}. To this aim we shall consider a situation where the density processes corresponding to the probability measures from $\cQ$ are related to a \textit{nearly sub-Gaussian family  of  local martingales} $(X^\theta~\dot=~(X_t^\theta)_{0 \leq t \leq T},\, \theta\in\Theta)$, i.e. each $X^{\theta}$ is a centered local martingale for which we assume that there is some $C\geq 1$ such that  
$$
\sup_{t\in [0,T]}\ex\left[\exp\big(\lambda (X_t^\theta- X_t^\vartheta\big)\right]\leq 
C\cdot\exp\big(\lambda^{2}~d_{\Theta}(\theta,\vartheta)^{2}/2\big)\quad\mbox{for}~\theta, \vartheta\in\Theta\quad\mbox{and}\quad \lambda>0.
$$
Especially this means that for fixed $t\in [0,T]$, any process $(X^{\theta}_{t})_{\theta\in\Theta}$ is a {\it nearly sub-Gaussian random field} in the sense considered in the Appendix.
In the case of $C = 1$ we end up with  the notion of sub-Gaussian families of local martingales.
The following result requires $d_{\Theta}$ to be totally bounded, and it relies on metric entropies w.r.t. $d_{\Theta}$. These are the numbers $\{\ln(N(\Theta,d_{\Theta};\varepsilon))\mid\varepsilon > 0\}$, where $N(\Theta,d_{\Theta};\varepsilon)$ denotes the minimal number of $\varepsilon$-balls needed to cover $\Theta$ w.r.t. $d_{\Theta}$. In addition, we define 
\[
{\cal D}(\delta,d_{\Theta})~\dot=~\int_{0}^{\delta}\sqrt{\ln(N(\Theta,d_{\Theta};\varepsilon))}~d\varepsilon.
\] 
\begin{proposition}
\label{prop_suffLLN3}
Let $\cQ$ satisfy \eqref{parametrization}, let the conditions \eqref{Integrierbarkeit} and \eqref{nonatomicdominated} be fulfilled
as well as $\sup_{\qr\in\cQ}\ex_{\qr}[Y^{*}\eins_{\{Y^{*} > a\}}]\to 0$ for $a\to\infty$. 
Furthermore, let $d_{\Theta}$ be totally bounded and let there exist a nearly sub-Gaussian family of local martingales $(X^\theta = (X_t^\theta)_{0 \leq t \leq T},\,\theta\in\Theta)$ such that the quadratic variation process 
$([X^\theta]_t)_{\theta\in\Theta}$ has $d_{\Theta}$-uniformly continuous paths for every $t \in [0,T]$, and such that the density processes of the probability measures from $\cQ$  may be represented in the following way 
\begin{equation}
\label{representation of M}
\left.\frac{d\qr_{\theta}}{d\pr}\right |_{\cF_t} = \exp(X_{t}^\theta - [X^\theta]_t/2)~\quad\mbox{pointwise for}~t \in [0,T]~\mbox{and}~\theta\in\Theta.
\end{equation}
If $\sup_{t\in [0,T]}\ex[\exp(2 X_t^{\overline{\theta}})] < \infty$  for some $\overline{\theta}\in\Theta$, and if   
${\cal D}(\Delta,d_{\Theta}) < \infty$, then 
$$
\sup_{\tau\in\cT}\, \inf_{\qr\in\cQ}\,\ex_{\qr}[Y_{\tau}] 
= \sup_{\tau\in\cT}\, \inf_{\qr\in co(\cQ)}\,\ex_{\qr}[Y_{\tau}] 
= \inf_{\qr\in co(\cQ)}\,\sup_{\tau\in\cT}\,\ex_{\qr}[Y_{\tau}].
$$
Moreover, $\cQ$ is not time-consistent w.r.t. $\OFFP$ if each of its members is equivalent to $\pr$, $\cQ$ has more than one element and  $L^{1}(\Omega,\cF_{t},\pr_{|\cF_{t}})$ is weakly separable for every $t > 0$.
\end{proposition}
The proof of Proposition \ref{prop_suffLLN3} may be found in Section \ref{LLN3}.
\medskip

The following example describes a typical situation where Proposition \ref{prop_suffLLN3} may be applied directly, and in particular representation \eqref{representation of M} occurs. 

\begin{example}
\label{stochastic integrals}
Let $Z~\dot=~(Z_s)_{s\geq 0}$ be a Brownian motion on $\OFPT$ such that $(Z_{t})_{0 \leq t \leq T}$ is adapted to $(\cF_t)_{0 \leq t \leq T}$, and let $V~\dot=~(V_t)_{0 \leq t \leq T}$ be some $\R^d$-valued process (volatility) adapted  to $(\cF_t)_{0 \leq t \leq T}$.
Consider a class $\Psi$ of Borel-measurable functions $\psi$: $[0,T] \times \mathbb{R}^d \rightarrow \mathbb{R}$ such that for every $\psi \in \Psi$ it holds that
\begin{align}
\label{psi-Bedingung}
\sup_{x \in \mathbb{R}^d} \int_0^T \psi^2(u, x)~du < \infty.
\end{align}
Then
\begin{align*}
d_{\Psi}: \Psi \times \Psi \rightarrow \mathbb{R},~
(\psi, \phi) \mapsto \sup_{x \in \mathbb{R}^d} \sqrt{\int_0^T (\psi - \phi)^2(u,x)~du}
\end{align*}
is a well-defined semimetric on $\Psi$. 
Assume that, for each $\psi \in \Psi,$ the process 
$(\psi(t, V_t))_{0 \leq t \leq T}$ is progressively measurable. 
Then the family of processes
$(X^\psi,\, \psi \in \Psi)$ with
\begin{align*}
X_t^\psi~\dot=~\int_0^t \psi(u, V_u)~dZ_u \quad (t\in [0,T])
\end{align*}
is well-defined and the quadratic variation process of $X^\psi$ is given by
\begin{eqnarray*}
[X^\psi]_t = \int_0^t\psi^{2}(u,V_u)~du \quad (t \in [0,T]).
\end{eqnarray*}
So by assumption \eqref{psi-Bedingung}, each process $\big(\exp(X^{\psi}_{t} - [X^{\psi}]_{t}/2)\big)_{t\in [0,T]}$ satisfies Novikov's condition. In particular it is even a martingale and it is the densitiy process of a probability measure on $\cF$ which is absolutely continuous w.r.t. $\pr$.
\medskip

For arbitrary $\overline{\psi}, \psi, \phi\in\Psi$ we may observe by Cauchy-Schwarz inequality for every $t \in [0,T]$
\begin{eqnarray*}
|[X^{\psi}]_t - [X^{\phi}]_t|
&\leq& 
\left(d_{\Psi}(\psi,\overline{\psi}) + d_{\Psi}(\phi,\overline{\psi})\right)\cdot d_{\Psi}(\psi,\phi).
\end{eqnarray*}
Hence $([(X^\psi)]_t)_{\psi \in \Psi}$ has $d_{\Psi}$-Lipschitz continuous paths whenever $d_{\Psi}$ is totally bounded.
Moreover, it can be shown that $(X^\psi,\, \psi \in \Psi)$ is a nearly sub-Gaussian family of  local martingales w.r.t. $d_{\Psi}$ with $C = 2$. The proof of this result can be found in Section~\ref{proof of stochastic integrals}. 

\end{example}
\medskip

\section{Discussion}
\label{discussion}
Let us discuss some related results in the literature. In \cite{KaratzasKou1998} and \cite{KaratzasZamfirescu2005} the minimax relationship \eqref{eq: minimax_intr} is studied for general convex sets $\cQ$ of probability measures which are equivalent to $\pr$ without explicitly imposing stability under pasting or time-consistency. However, it is  implicitly assumed  there (see \cite[proof of Lemma B.1]{KaratzasKou1998} and \cite[proof of Proposition 3.1]{KaratzasZamfirescu2005}) that one can
find, for every $\tau\in\cT$ and any $\overline{\qr}\in\cQ,$ a sequence $(\qr^{k})_{k\in\N}$ of probability measures from $\cQ$ which agree with $\overline \qr$ on $\cF_{\tau}$ such that
$$
\esssup_{\sigma\in\cT,\sigma\geq \tau}~
\ex_{\qr^{k}}[Y_{\tau}|\cF_{\tau}] \xrightarrow[k\to\infty]{ }\essinf_{\qr\in\cQ}\esssup_{\sigma\in\cT, \sigma\geq \tau}\ex_\qr{[Y_{\sigma}|\cF_{\tau}]}\quad\pr\mbox{-a.s.}.
$$
It turns out that  only for time-consistent sets $\cQ$ the above relation can hold in general. 
\begin{proposition}
\label{bereits Zeitkonsistenz}
Let $\widehat{Q}$ denote the set of all probability measures on $\cF$ such that 
$\ex_{\qr}[X]\geq\inf_{\qr'\in\cQ}\ex_{\qr'}[X]$ holds for every $\pr$-essentially bounded random variable $X$. Furthermore let each member of $\cQ$ be equivalent to $\pr$, and define the set $\cS(\cQ)$ to consist of all uniformly bounded adapted càdlàg processes $Z = (Z_{t})_{0 \leq t \leq T}$ such that every of the single stopping problems
$$
\sup_{\tau\in\cT}\ex_{\qr}[Z_{\tau}]\quad (\qr\in\cQ)
$$
has a solution. Consider the following statements:
\begin{itemize}
\item [(1)] $\cQ$ is time-consistent.
\item [(2)] $\inf_{\qr\in\cQ}\ex_{\qr}[X]\leq \inf_{\qr\in\cQ}\ex_{\qr}\left[\essinf_{\qr\in\cQ}~\ex_{\qr}[X|\cF_{\tau}]\right]$ holds for every $\pr$-essentially bounded random variable $X$ and every stopping time $\tau\in\cT$.
\item [(3)] $\widehat{Q}^{e}~\dot=~\{\qr\in\widehat{\cQ}\mid \qr\approx\pr\}$ is stable under pasting, and 
$$
\essinf_{\qr\in\cQ}~\ex_{\qr}[X|\cF_{\tau}] = \essinf_{\qr\in\widehat{\cQ}^{e}}~\ex_{\qr}[X|\cF_{\tau}]\quad\mbox{for}~\pr\mbox{-essentially bounded}~X,\tau\in\cT.
$$
\item [(4)] For arbitrary process $Z = (Z_{t})_{0 \leq t \leq T}\in\cS(\cQ)$, and for any $\tau\in\cT$ as well as $\overline{\qr}\in\cQ$, there is some sequence $(\qr^{k})_{k\in\N}$ in $\cQ$ whose members agree with $\overline{\qr}$ on $\cF_{\tau}$ such that
$$
\esssup_{\sigma\in\cT,\sigma\geq\tau}~\ex_{\qr_{k}}[Z_{\sigma}|\cF_{\tau}] \xrightarrow[k\to\infty]{ } \essinf_{\qr\in\cQ}~\esssup_{\sigma\in\cT,\sigma\geq\tau}~\ex_{\qr}[Z_{\sigma}|\cF_{\tau}]\quad\pr\mbox{-a.s.}.
$$
\item [(5)] $\cQ$ is stable under pasting.
\end{itemize}
Then the statements $(1) - (3)$ are equivalent and $(4)$ follows from $(5)$. Moreover, the implication $(4)\Rightarrow (1)$ holds. 
\end{proposition}
The proof of Proposition \ref{bereits Zeitkonsistenz} is delegated to Subsection \ref{Zeitkonsistenz}.
\section{An abstract minimax result for lower Snell envelopes}
\label{general abstract minimax}
Let us define the set $\cX$ of all random variables $X$ on $(\Omega,\cF,\pr)$ satisfying
$$
|X|\leq C\big(Y^{*} + 1\big)~\pr\mbox{-a.s. for some}~C > 0.
$$
Note that $\cX$ is a Stonean vector lattice enclosing the set $\{Y_{\tau}\mid\tau\in\cT\}$ and the space $L^{\infty}(\Omega,\cF,\pr)$ of all $\pr$-essentially bounded random variables. Moreover, $\cX\subseteq L^{1}(\cQ)$ is valid under assumption \eqref{Integrierbarkeit}, and in this case we may introduce the following mapping
$$
\rho_{\cQ}:\cX\rightarrow\R,~X\mapsto\sup_{\qr\in\cQ}\ex_{\qr}[X].
$$
We shall call $\rho_{\cQ}$ to be {\it continuous from above at $0$} if $\rho_{\cQ}(X_{n})\searrow 0$ for $X_{n}\searrow 0$ $\pr$-a.s..

\medskip

In this section we want to present a general abstract minimax relation \eqref{eq: minimax_intr} which will be the starting point to derive the main result Theorem \ref{secondMinimax}.
It relies on the following key assumption.
\begin{enumerate}
\item [(A)] There exists some $\lambda\in ]0,1[$ such that for every $\tau_{1},\tau_{2}\in\cT_{f}\setminus\{0\}$ 
$$
\inf_{A\in\cF_{\tau_{1}\wedge\tau_{2}}}\rho_{\cQ}((\eins_{A} - \lambda) (Y_{\tau_{2}} - Y_{\tau_{1}}))\leq 0,
$$
\end{enumerate}
where $\cT_{f}$ denotes the set of stopping times from $\cT$ with finite range.
\begin{theorem}
\label{generalminimax}
If $Y = (Y_{t})_{0 \leq t \leq T}$ fulfills \eqref{Integrierbarkeit}, and if $\rho_{\cQ}$ continuous from above at $0$, then under assumption {\rm (A)}
$$
\sup_{\tau\in\cT}\, \inf_{\qr\in\cQ}\,\ex_{\qr}[Y_{\tau}] 
= 
\sup_{\tau\in\cT}\, \inf_{\qr\in co(\cQ)}\,\ex_{\qr}[Y_{\tau}] 
=
\inf_{\qr\in co(\cQ)}\,\sup_{\tau\in\cT}\,\ex_{\qr}[Y_{\tau}].
$$
\end{theorem}
The proof of Theorem \ref{generalminimax} is delegated to Section \ref{proof of generalminimax}. 
\medskip

At this place we may invoke the assumption of Theorem \ref{secondMinimax} that the range of the vector measure $\mu_{\cQ}$ associated with $\cQ$ should be relatively compact w.r.t. the sup-norm $\|\cdot\|_{\infty}$. As the following result shows this condition essentially implies assumption {\rm (A)}.
\begin{proposition}
\label{Basiskriterium}
Let $Y = (Y_{t})_{0 \leq t \leq T}$ satisfy \eqref{Integrierbarkeit}, and let $\rho_{\cQ}$ be continuous from above at $0$. Suppose furthermore that 
$\{\mu_{\cQ}(A)\mid A\in\cF\}$ is relatively $\|\cdot\|_{\infty}$-compact. If for $\tau_{1},\tau_{2}\in\cT_f\setminus\{0\}$ the probability space $(\Omega,\cF_{\tau_{1}\wedge\tau_{2}},\pr_{|\cF_{\tau_{1}\wedge\tau_{2}}})$ is atomless, then
$$
\inf_{A\in\cF_{\tau_{1}\wedge\tau_{2}}}\rho_{\cQ}((\eins_{A} - 1/2) (Y_{\tau_{2}} - Y_{\tau_{1}}))\leq 0.
$$
\end{proposition}
The proof of Proposition \ref{Basiskriterium} may be found in Section \ref{proof Basiskriterium}.
\section{Proofs}
\label{proofs}
Let \eqref{Integrierbarkeit} be fulfilled. Note that under \eqref{Integrierbarkeit}
\begin{equation}
\label{Wohldefiniertheit}
Y_{\tau}\in L^{1}(\cQ)\quad\mbox{for}~\tau\in\cT. 
\end{equation}
Condition \eqref{Wohldefiniertheit} implies that for any $\qr\in co(\cQ)$ its Radon-Nikodym derivative $\frac{d\qr}{d\pr}$ satisfies
\begin{equation}
\label{integrierbare Ableitungen}
Y_{\tau}~\frac{d\qr}{d\pr}~\mbox{is}~\pr\mbox{-integrable for every}~\tau\in\cT.
\end{equation}
Let the set $\cX$ and the mapping $\rho_{\cQ}$ be defined as at the beginning of section \ref{general abstract minimax}.
\subsection{A topological closure of {\bf $\cQ$}}
\label{spezielle topological closure}
Let $\cM_{1}(\Omega,\cX)$ denote the set of probability measures $\qr$ on $\cF$ such that $X$ is $\qr$-integrable for every $X\in \cX$. Set 
$$
\overline{\cQ}~\dot=~\big\{\qr\in\cM_{1}(\Omega,\cX)\mid \sup_{X\in \cX}\big(\ex_{\qr}[X] - \rho_{\cQ}(X)\big)\leq 0\big\}.
$$
Obviously $co(\cQ)\subseteq\overline{\cQ}$, and
\begin{equation}
\label{Vorbereitungen}
\overline{\cQ}~\mbox{is convex}\quad\mbox{and}\quad\sup_{\qr\in\overline{\cQ}}\ex_{\qr}[X] = \rho_{\cQ}(X)~\mbox{for all}~X\in \cX.
\end{equation}
We endow $\overline{\cQ}$ with the coarsest topology $\sigma(\overline{\cQ},\cX)$ such that the mappings
$$
\varphi_{X}:\overline{\cQ}\rightarrow\R,~\qr\mapsto\ex_{\qr}[X]\quad(X\in \cX)
$$
are continuous. In the next step we are going to investigate when $\sigma(\overline{\cQ},\cX)$ is compact with $co(\cQ)$ being a dense subset.
\begin{lemma}
\label{compactness condition}
If \eqref{Integrierbarkeit} holds, and if $\rho_{\cQ}$ is continuous from above at $0$, then $\sigma(\overline{\cQ},\cX)$ is compact. Moreover, 
$co(\cQ)$ is a $\sigma(\overline{\cQ},\cX)$-dense subset of $\overline{\cQ}$, and $\overline{\cQ}$ is dominated by $\pr$.
\end{lemma}
\begin{proof}
Let us equip the algebraic dual $\cX^{*}$ of $\cX$ with the coarsest topology $\sigma (\cX^{*},\cX)$ such that the mappings
$$
h_{X}: \cX^{*}\rightarrow\R,~\Lambda\mapsto\Lambda(X)\quad(X\in \cX)
$$
are continuous. The functional $\rho_{\cQ}$ is sublinear. Then by a version of the Banach Alaoglu theorem (cf. \cite[Theorem 1.6]{Koenig2001}) the set
\begin{eqnarray*}
\Delta_{\cQ} 
~\dot=~
\big\{\Lambda\in L^{1}(\cQ)^{*}\mid \sup_{X\in L^{1}(\cQ)}(\Lambda(X) - \rho_{\cQ}(X))\leq 0\big\} 
\end{eqnarray*}
is compact w.r.t. $\sigma(\cX^{*},\cX)$. Moreover, $\rho_{\cQ}$ is assumed to be continuous from above at $0$. This implies that every $\Lambda\in\Delta_{\cQ}$ satisfies $\Lambda(X_{n})\searrow 0$ whenever $X_{n}\searrow 0$. Since $\cX$ is a Stonean vector lattice containing the constant mappings on $\Omega$, it generates the $\sigma$-algebra $\cF$, and the application of the Daniell-Stone representation theorem yields that each $\Lambda\in cl(\{\Lambda_{\qr}\mid \qr\in\cQ\})$ is uniquely representable by a probability measure $\qr_{\Lambda}$, namely
$$
\qr_{\Lambda}:\cF\rightarrow [0,1],~A\mapsto\Lambda(\eins_{A}).
$$
Hence by definition of $\overline{\cQ}$ we obtain
\begin{equation}
\label{DaniellStoneRepresentation}
\Delta_{\cQ} = \{\Lambda_{\qr}\mid \qr\in\overline{\cQ}\},\quad\mbox{and}\quad\Lambda_{\qr}\not=\Lambda_{\widetilde{\qr}}~\mbox{for}~\qr\not=\widetilde{\qr},
\end{equation}
where 
$$
\Lambda_{\qr}: \cX\rightarrow\R,~X\mapsto \ex_{\qr}[X]\quad\mbox{for}~\qr\in\overline{\cQ}.
$$
Obviously we may define a homeomorphism from $\Delta_{\cQ}$ onto 
$\overline{\cQ}$ w.r.t. the topologies $\sigma(\cX^{*},\cX)$ and 
$\sigma(\overline{\cQ},\cX)$. In particular, $\overline{\cQ}$ is compact w.r.t. $\sigma(\cX^{*},\cX)$.

\medskip

Next, $\{\Lambda_{\qr}\mid \qr\in co(\cQ)\}$
is a convex subset of $\cX^{*}$. We may draw on a version of the Bipolar theorem 
(cf. \cite[Consequence 1.5]{Koenig2001}) to observe that the $\sigma(\cX^{*},\cX)$-closure $cl(\{\Lambda_{\qr}\mid \qr\in co(\cQ)\})$ of $\{\Lambda_{\qr}\mid \qr\in co(\cQ)\}$ coincides with $\Delta_{\cQ}$. Therefore \eqref{DaniellStoneRepresentation} enables us to define a homeomorphism from $cl(\{\Lambda_{\qr}\mid \qr\in co(\cQ)\})$ onto 
$\overline{\cQ}$ w.r.t. the topologies $\sigma(\cX^{*},\cX)$ and 
$\sigma(\overline{\cQ},\cX)$. Thus $co(\cQ)$ is a $\sigma(\cX^{*},\cX)$-dense subset of $\overline{\cQ}$, and by definition of the topology $\sigma(\overline{\cQ},\cX)$ it may be verified easily that $\overline{\cQ}$ is dominated by $\pr$.  This completes the proof.
\end{proof}
Consider now the following new optimization problems
\begin{equation}
\label{stoppproblemneu}
\mbox{maximize}~\inf_{\qr\in\overline{\cQ}}\,\ex_{\qr}[Y_{\tau}]\quad\mbox{over}~\tau\in\cT,
\end{equation}
and
\begin{equation}
\label{stoppproblemdualneu}
\mbox{minimize}~\sup_{\tau\in\cT}\,\ex_{\qr}[Y_{\tau}]\quad\mbox{over}~\qr\in\overline{\cQ}.
\end{equation}
In view of \eqref{Vorbereitungen} we obtain that 
\eqref{stoppproblemneu} has the same optimal value as the corresponding one w.r.t. $\cQ$ and $co(\cQ)$.
\begin{proposition}
\label{same optimal value}
Under the assumption \eqref{Integrierbarkeit} we have
$$
\sup_{\tau\in\cT}\inf_{\qr\in\overline{\cQ}}\,\ex_{\qr}[Y_{\tau}] 
= 
\sup_{\tau\in\cT}\inf_{\qr\in\cQ}\,\ex_{\qr}[Y_{\tau}]
= 
\sup_{\tau\in\cT}\inf_{\qr\in co(\cQ)}\,\ex_{\qr}[Y_{\tau}].
$$
\end{proposition}
The comparison of the optimal value 
of problem \eqref{stoppproblemdualneu} with the corresponding one w.r.t. $co(\cQ)$ is more difficult to handle. For preparation let us introduce a sequence 
$(\oYk)_{k\in\N}$ of stochastic processes $\oYk~\dot=~(\oYk_{t})_{0 \leq t \leq T}$ via $\oYk_{t}~\dot=~(Y_{t}\wedge k)\vee (-k)$. They all are adapted to $(\cF_t)_{0 \leq t \leq T}$.
\begin{lemma}
\label{boundedApproximation}
If \eqref{Integrierbarkeit} is satisfied, and if $\rho_{\cQ}$ is continuous from above at $0$, then
$$
\lim_{k\to\infty}\sup_{\qr\in\overline{\cQ}}\big|\sup_{\tau\in\cT}\ex_{\qr}[Y_{\tau}] - \sup_{\tau\in\cT}\ex_{\qr}[\oYk_{\tau}]\big| = 0.
$$
\end{lemma}
\begin{proof}
For $\tau\in\cT$ and $k\in\N$ we may observe $Y^{*}\in\cX$ by \eqref{Integrierbarkeit} and
$$
|Y_{\tau} - \oYk_{\tau}|\leq \eins_{\{Y^{*} > k\}}
\big(Y^{*} - k\big).
$$
Then
\begin{eqnarray*}
\sup_{\qr\in\overline{\cQ}}\big|\sup_{\tau\in\cT}\ex_{\qr}[Y_{\tau}] - \sup_{\tau\in\cT}\ex_{\qr}[\oYk_{\tau}]\big|
&\leq&
\sup_{\qr\in\overline{\cQ}}\sup_{\tau\in\cT}\ex_{\qr}[|Y_{\tau} - \oYk_{\tau}|]\\
&\leq& 
\sup_{\qr\in\overline{\cQ}}\ex_{\qr}\big[\eins_{\{Y^{*} > k\}}
\big(Y^{*} - k\big)\big]\\
&\stackrel{\eqref{Vorbereitungen}}{=}&
\rho_{\cQ}\Big(\eins_{\{Y^{*} > k\}}\cdot
\big(Y^{*} - k\big)\Big)
\end{eqnarray*}
holds for any $k\in\N$. Finally, $\eins_{\{Y^{*} > k\}}\cdot
\big(Y^{*} - k\big)\searrow 0$, and thus 
the statement of Lemma \ref{boundedApproximation} follows immediately because $\rho_{\cQ}$ is assumed to be continuous from above at $0$. 
\end{proof}
In the next providing step we shall replace in \eqref{stoppproblemdualneu} the process $Y$ with the processes $\oYk$.
We want to look when the optimal value coincides with optimal value of the corresponding problem w.r.t. $co(\cQ)$.
\begin{lemma}
\label{same optimal value bounded}
If \eqref{Integrierbarkeit} holds, 
and if $\rho_{\cQ}$ is continuous from above at $0$, then
$$
\inf_{\qr\in\overline{\cQ}}\sup_{\tau\in\cT}\ex_{\qr}[\oYk_{\tau}] = 
\inf_{\qr\in co(\cQ)}\sup_{\tau\in\cT}\ex_{\qr}[\oYk_{\tau}]\quad\mbox{for every}~k\in\N.
$$
\end{lemma}
\begin{proof}
Let $k\in\N$, and fix $\qr_{0}\in\overline{\cQ}$. In view of Lemma \ref{compactness condition} its Radon-Nikodym derivative $d\qr_{0}/d\pr$ is in the weak closure of $\{d\qr/d\pr\mid \qr\in co(\cQ)\}$, viewed as a subset of the $L^{1}$-space on $(\Omega,\cF,\pr)$. Moreover, by Lemma \ref{compactness condition}, the set $\{d\qr/d\pr\mid \qr\in co(\cQ)\}$ is a relatively weakly compact subset of the $L^{1}$-space on $(\Omega,\cF,\pr)$. Then by Eberlein-Smulian theorem (cf. e.g. \cite{Kremp1986}), we may select a sequence $(\qr_{n})_{n\in\N}$ in $co(\cQ)$ such that
\begin{equation}
\label{Eberlein-Smulian}
\lim_{n\to\infty}\ex_{\qr_{n}}[X] 
=
\lim_{n\to\infty}\ex\Big[X~\frac{d\qr_{n}}{d\pr}\Big] 
=
\ex\Big[X~\frac{d\qr_{0}}{d\pr}\Big] 
=
\ex_{\qr_{0}}[X]
\end{equation}
holds for every $\pr$-essentially bounded random variable $X$.
\medskip

Let us introduce the set $\cP^{e}$ of probability measures on $\cF$ which are equivalent to $\pr$. Then for any $\qr\in\cP^{e}$ the $\sigma$-algebra $\cF_{0}$ contains all null sets of $\cF$, and $\qr(A)\in\{0,1\}$ holds for every $A\in\cF_{0}$. In particular, the set $\cT^{\qr}$ of all stopping times w.r.t. $\OFFQ$ coincides with $\cT$ for $\qr\in\cP^{e}$. Note also that $(\oYk_{t} + k)_{0 \leq t \leq T}$ is a nonnegative, bounded, right-continuous and $(\cF_t)$-adapted process which is quasi left-uppersemicontinuous w.r.t. every 
$\qr\in\cP^{e}$.
Here quasi left-uppersemicontinuity w.r.t. $\qr$ has to be understood as we have defined it w.r.t. to $\pr$. Hence, by  Fatou lemma $\limsup_{m\to\infty}\ex_{\qr}[\oYk_{\tau_{m}}+k]\leq \ex_{\qr}[\oYk_{\tau}+k]$ is valid for every $\qr\in\cP^{e}$ whenever $(\tau_{m})_{m\in\N}$ is a sequence in $\cT$ satisfying $\tau_{m}\nearrow\tau$ for some $\tau\in\cT$. 
Then we may draw on \cite[Proposition B.6]{KobylanskiQuenez2012} to conclude
\begin{equation}
\label{optimal stopping time}
\forall~\qr\in\cP^{e}~\exists~\tau\in\cT:~\ex_{\qr}[\oYk_{\tau} + k] = \sup_{\tau\in\cT}\ex_{\qr}[\oYk_{\tau} + k]. 
\end{equation}
Let us denote for $\qr\in co(\cQ)$ and $\lambda\in ]0,1[$ by $\qr^{\lambda}$ the probability measure on $\cF$ defined via $\pr$-Radon-Nikodym derivative
$$
\frac{d\qr^{\lambda}}{d\pr}~\dot=~\lambda \frac{d\qr}{d\pr} + (1-\lambda).
$$
By this construction we may introduce the sets 
$$
\cQ^{\lambda}~\dot=~\{\qr^{\lambda}\mid \qr\in co(\cQ)\}\quad(\lambda \in ]0,1[).
$$
Obviously, these sets are contained in $\cP^{e}$. Now, define for $\lambda\in ]0,1[$ the sequence $(f_{n,\lambda})_{n\in\N}$ of mappings 
$$
f_{n,\lambda}:\cT\rightarrow\R,~\tau\mapsto\ex_{\qr_{n}^{\lambda}}[\oYk_{\tau} + k].
$$
Notice that the sequence $(f_{n,\lambda})_{n\in\N}$ is uniformly bounded for $\lambda\in ]0,1[$ because 
\begin{equation}
\label{gleichmaessige Beschraenkung}
|\oYk_{\tau} + k|\leq 2 k\quad\mbox{for every}~\tau\in\cT.
\end{equation}
We want to apply Simons' lemma (cf. \cite[Lemma 2]{Simons1972}) to each sequence $(f_{n,\lambda})_{n\in\N}$. For this purpose it is left to show for fixed $\lambda\in ]0,1[$ that we may find for any countable convex combination of $(f_{n,\lambda})_{n\in\N}$ some maximizer. So let $(\lambda_{n})_{n\in\N}$ be a sequence in $[0,1]$ with $\sum_{n=1}^{\infty}\lambda_{n} = 1$. We may define by 
$$
\sum_{n=1}^{\infty}\lambda_{n}\qr^{\lambda}_{n}(A) = \qr(A)\quad\mbox{for every}~A\in\cF
$$ 
a probability measure on $\cF$ which belongs to $\cP^{e}$. Then by monotone convergence theorem
\begin{eqnarray}
\nonumber
\sum_{n=1}^{\infty}\lambda_{n}f_{n,\lambda}(\tau) 
= 
\sum_{n=1}^{\infty}\lambda_{n}\int_{0}^{\infty}\qr^{\lambda}_{n}(\{(\oYk_{\tau} + k) > x\})~dx
&=&
\int_{0}^{\infty}\qr(\{(\oYk_{\tau} + k) > x\})~dx\\
\label{Basisdarstellung} 
&=& 
\ex_{\qr}[\oYk_{\tau} + k]\quad\mbox{for}~\tau\in\cT.
\end{eqnarray}
Moreover, by \eqref{optimal stopping time}, there exists some $\tau_{*}\in\cT$ such that
$$
\sum_{n=1}^{\infty}\lambda_{n}f_{n, \lambda}(\tau_{*}) 
= 
\ex_{\qr}[\oYk_{\tau_{*}} + k] 
= 
\sup_{\tau\in\cT}\ex_{\qr}[\oYk_{\tau} + k] 
= 
\sup_{\tau\in\cT}\sum_{n=1}^{\infty}\lambda_{n}f_{n, \lambda}(\tau) 
$$
Therefore, the assumptions of Simons' lemma (cf. \cite[Lemma 2]{Simons1972}) are satisfied so that we may conclude
$$
\sup_{\tau\in\cT}\limsup_{n\to\infty}f_{n,\lambda}(\tau)
\geq 
\inf_{f\in co(\{f_{n,\lambda}\mid n\in\N\})}\sup_{\tau\in\cT}f(\tau),
$$
where $co(\{f_{n;\lambda}\mid n\in\N\})$ denotes the convex hull of $\{f_{n,\lambda}\mid n\in\N\}$. For any convex combination 
$f = \sum_{i=1}^{r}\lambda_{i} f_{n_{i},\lambda}$, the probability measure $\qr~\dot=~\sum_{i=1}^{r}\lambda_{i}\qr_{n_{i}}$ is a member of $co(\cQ)$, 
and 
$$
f(\tau) = \ex_{\qr^{\lambda}}[\oYk_{\tau} + k]\quad\mbox{for}~\tau\in\cT.
$$
Therefore on the one hand
$$
\sup_{\tau\in\cT}\limsup_{n\to\infty}\ex_{\qr^{\lambda}_{n}}[\oYk_{\tau} + k]
\geq 
\inf_{\qr\in\cQ^{\lambda}}\sup_{\tau\in\cT}\ex_{\qr}[\oYk_{\tau} + k].
$$
On the other hand by \eqref{Eberlein-Smulian}
\begin{eqnarray*}
\sup_{\tau\in\cT}\limsup_{n\to\infty}\ex_{\qr^{\lambda}_{n}}[\oYk_{\tau} + k] 
= 
\sup_{\tau\in\cT}\left(\lambda\ex_{\qr_{0}}[\oYk_{\tau} + k] + (1-\lambda)\ex[\oYk_{\tau} + k]\right).
\end{eqnarray*}
Hence by \eqref{gleichmaessige Beschraenkung} and nonnegativity of $(\oYk_{t} + k)_{0 \leq t \leq T}$
\begin{eqnarray*}
\lambda \sup_{\tau\in\cT}\ex_{\qr_{0}}[\oYk_{\tau} + k] + (1- \lambda) 2 k
&\geq&
\sup_{\tau\in\cT}\left(\lambda\ex_{\qr_{0}}[\oYk_{\tau} + k] + (1-\lambda)\ex[\oYk_{\tau} + k]\right)\\
&\geq& 
\inf_{\qr\in\cQ^{\lambda}}\sup_{\tau\in\cT}\ex_{\qr}[\oYk_{\tau} + k]\\
&=& 
\inf_{\qr\in co(\cQ)}\sup_{\tau\in\cT}\left(\lambda \ex_{\qr}[\oYk_{\tau} + k] + (1-\lambda)\ex[\oYk_{\tau} + k]\right)\\
&\geq&
\lambda \inf_{\qr\in co(\cQ)}\sup_{\tau\in\cT}\ex_{\qr}[\oYk_{\tau} + k].
\end{eqnarray*}
Then by sending $\lambda\nearrow 1$
\begin{eqnarray*}
\sup_{\tau\in\cT}\ex_{\qr_{0}}[\oYk_{\tau} + k] 
\geq
\inf_{\qr\in co(\cQ)}\sup_{\tau\in\cT}\ex_{\qr}[\oYk_{\tau} + k],
\end{eqnarray*}
and thus
$$
\sup_{\tau\in\cT}\ex_{\qr_{0}}[\oYk_{\tau}] 
\geq 
\inf_{\qr\in co(\cQ)}\sup_{\tau\in\cT}\ex_{\qr}[\oYk_{\tau}].
$$
This completes the proof because $\qr_{0}$ was arbitrarily chosen from $\overline{\cQ}$, and $co(\cQ)\subseteq\overline{\cQ}$ holds.
\end{proof}
We are ready to provide the following criterion which ensures that the optimal value 
of \eqref{stoppproblemdualneu} coincides with the optimal value of the corresponding problem w.r.t. $co(\cQ)$.
\begin{proposition}
\label{same dual optimal value}
Let \eqref{Integrierbarkeit} be fulfilled. If 
$\rho_{\cQ}$ is continuous from above at $0$, then
$$
\inf_{\qr\in\overline{\cQ}}\sup_{\tau\in\cT}\ex_{\qr}[Y_{\tau}] = 
\inf_{\qr\in co(\cQ)}\sup_{\tau\in\cT}\ex_{\qr}[Y_{\tau}].
$$
\end{proposition}
\begin{proof}
Firstly $\inf_{\qr\in\overline{\cQ}}\sup_{\tau\in\cT}\ex_{\qr}[Y_{\tau}] \leq 
\inf_{\qr\in co(\cQ)}\sup_{\tau\in\cT}\ex_{\qr}[Y_{\tau}]$ because $co(\cQ)\subseteq\overline{\cQ}$. Then we obtain by Lemma \ref{same optimal value bounded} for every $k\in\N$
\begin{eqnarray*}
0
&\leq&
\inf_{\qr\in co(\cQ)}\sup_{\tau\in\cT}\ex_{\qr}[Y_{\tau}] -
\inf_{\qr\in\overline{\cQ}}\sup_{\tau\in\cT}\ex_{\qr}[Y_{\tau}] \\
&\stackrel{{\rm Lemma~ \ref{same optimal value bounded}}}{=}& 
\inf_{\qr\in co(\cQ)}\sup_{\tau\in\cT}\ex_{\qr}[Y_{\tau}] -
\inf_{\qr\in co(\cQ)}\sup_{\tau\in\cT}\ex_{\qr}[\oY^{k}_{\tau}] +
\inf_{\qr\in\overline{\cQ}}\sup_{\tau\in\cT}\ex_{\qr}[\oY^{k}_{\tau}] -
\inf_{\qr\in\overline{\cQ}}\sup_{\tau\in\cT}\ex_{\qr}[Y_{\tau}] \\
&\leq& 
2\cdot\sup_{\qr\in\overline{\cQ}}~\big|~\sup_{\tau\in\cT}\ex_{\qr}[Y_{\tau}] - 
\sup_{\tau\in\cT}\ex_{\qr}[\oY^{k}_{\tau}]~\big|
\end{eqnarray*}
The statement of Proposition \ref{same dual optimal value} follows now immediately from 
Lemma \ref{boundedApproximation}.
\end{proof}

\subsection{Proof of Theorem \ref{generalminimax}}
\label{proof of generalminimax}
Let $\overline{\cQ}$ be defined as in the previous subsection. The idea of the proof is to verify first duality of the problems 
\eqref{stoppproblemneu} and \eqref{stoppproblemdualneu}, and then to apply Proposition \ref{same optimal value} along with Proposition \ref{same dual optimal value}. Concerning the minimax relationship of the problems 
\eqref{stoppproblemneu} and \eqref{stoppproblemdualneu} we may reduce considerations to stopping times with finite range if $\rho_{\cQ}$ is continuous from above at $0$.
\begin{lemma}
\label{finite range stopping times}
If $Y$ fulfills \eqref{Integrierbarkeit}, and if $\rho_{\cQ}$ is continuous from above at 0, then
\begin{enumerate}
\item [{\rm (i)}] $\sup_{\tau\in\cT}\inf_{\qr\in\overline{\cQ}}\ex_{\qr}[Y_{\tau}] = \sup_{\tau\in\cT_{f}}\inf_{\qr\in\overline{\cQ}}\ex_{\qr}[Y_{\tau}]$;
\item [{\rm (ii)}] $\inf_{\qr\in\overline{\cQ}}\sup_{\tau\in\cT}\ex_{\qr}[Y_{\tau}] = \inf_{\qr\in\overline{\cQ}}\sup_{\tau\in\cT_{f}}\ex_{\qr}[Y_{\tau}]$.
\end{enumerate}
Here $\cT_f$ denotes the set of all stopping times from $\cT$ with finite range.
\end{lemma}
\begin{proof}
For $\tau\in\cT$ we may define by
$$
\tau[j](\omega) := \min\{k/2^{j}\mid k\in\N, \tau(\omega)\leq k/2^{j}\}\wedge T
$$
a sequence $(\taur[j])_{j\in\N}$ in $\cT_{f}$ satisfying the 
$\tau[j]\searrow\tau$ pointwise, and by right-continuity of the paths of $Y$
\begin{equation}
\label{Diskretisierung}
\lim\limits_{j\to\infty}Y_{\tau[j](\omega)}(\omega) = Y_{\tau(\omega)}(\omega)\quad\mbox{for any}~\omega\in\Omega. 
\end{equation}
For the proof of statement (i) let us fix any $\tau\in\cT$. Then $|Y_{\tau} - Y_{\tau[j]}|\to 0$ pointwise for $j\to\infty$ due to \eqref{Diskretisierung}. Set 
$$
\hY_{k}~\dot=~\sup_{j\geq k}|Y_{\tau} - Y_{\tau[j]}|\quad\mbox{for}~k\in\N.
$$
This defines a sequence $(\hY_{k})_{k\in\N}$ of random variables $\hY_{k}$ on $(\Omega,\cF,\pr)$ which satisfy 
$|\hY_{k}|\leq 2\sup_{t\in [0,T]}|Y_{t}|$ so that they belong to $L^{1}(\cQ)$. Since 
$\hY_{k}\searrow 0$, and since $\rho_{\cQ}$ is continuous from above at 0, we obtain
$$
0\leq\rho_{\cQ}(|Y_{\tau} - Y_{\tau[j]}|)\leq\rho_{\cQ}(\hY_{j})\to 0\quad\mbox{for}~j\to\infty.
$$
Hence by \eqref{Vorbereitungen}
$$
0 \leq |\inf_{\qr\in\overline{\cQ}}\ex_{\qr}[Y_{\tau}] - \inf_{\qr\in\overline{\cQ}}\ex_{\qr}[Y_{\tau[j]}]|
\stackrel{\eqref{Vorbereitungen}}{\leq} 
\rho_{\cQ}(|Y_{\tau} - Y_{\tau[j]}|)\to 0 \quad\mbox{for}~j\to\infty,
$$
and thus
$$
\sup_{\tau\in\cT_{f}}\inf_{\qr\in\overline{\cQ}}\ex_{\qr}[Y_{\tau}]
\geq
\lim_{j\to\infty}\inf_{\qr\in\overline{\cQ}}\ex_{\qr}[Y_{\tau[j]}] 
= 
\inf_{\qr\in\overline{\cQ}}\ex_{\qr}[Y_{\tau}].
$$
The stopping time $\tau$ was arbitrarily chosen so that we may conclude
$$
\sup_{\tau\in\cT_{f}}\inf_{\qr\in\overline{\cQ}}\ex_{\qr}[Y_{\tau}] 
\geq 
\sup_{\tau\in\cT}\inf_{\qr\in\overline{\cQ}}\ex_{\qr}[Y_{\tau}]
\geq 
\sup_{\tau\in\cT_{f}}\inf_{\qr\in\overline{\cQ}}\ex_{\qr}[Y_{\tau}],
$$
where the last inequality is obvious due to $\cT_{f}\subseteq\cT$. So statement (i) is shown.
\bigskip

In order to prove statement (ii) let us fix any $\varepsilon > 0$.  Then for arbitrary $\qr\in\overline{\cQ}$ we may find some $\tau_{0}\in\cT$ such that
\begin{equation}
\label{erste Ungleichung}
\sup_{\tau\in\cT}\ex_{\qr}[Y_{\tau}] - \varepsilon < \ex_{\qr}[Y_{\tau_{0}}].
\end{equation}
We have $Y_{\tau_{0}[j]}\to Y_{\tau_{0}}$ pointwise for $j\to\infty$, and 
$|Y_{\tau_{0}[j]}|\leq \sup_{t\in [0,T]}|Y_{t}|$ for every $j\in\N$. So in view of 
\eqref{Integrierbarkeit} we may apply the dominated convergence theorem to conclude
$$
\lim_{j\to\infty}\ex_{\qr}[Y_{\tau_{0}[j]}] = \ex_{\qr}[Y_{\tau_{0}}],
$$
and thus by \eqref{erste Ungleichung}
$$
\sup_{\tau\in\cT_{f}}\ex_{\qr}[Y_{\tau}]\geq \lim_{j\to\infty}\ex_{\qr}[Y_{\tau_{0}[j]}] = \ex_{\qr}[Y_{\tau_{0}}] > \sup_{\tau\in\cT}\ex_{\qr}[Y_{\tau}] - \varepsilon.
$$
Letting $\varepsilon\searrow 0$, we obtain
$$
\sup_{\tau\in\cT_{f}}\ex_{\qr}[Y_{\tau}]\geq \sup_{\tau\in\cT}\ex_{\qr}[Y_{\tau}]
\geq 
\sup_{\tau\in\cT_{f}}\ex_{\qr}[Y_{\tau}],
$$
where the last inequality is trivial due to $\cT_{f}\subseteq\cT$. Since $\qr$ was arbitrarily chosen, statement (ii) follows immediately. The proof is complete. 
\end{proof}
In the next step we want to show $\sup_{\tau\in\cT_{f}}\inf_{\qr\in\overline{\cQ}}\ex_{\qr}[Y_{\tau}] = \inf_{\qr\in\overline{\cQ}}\sup_{\tau\in\cT_{f}}\ex_{\qr}[Y_{\tau}]$.
\begin{proposition}
\label{Hilfsminimax}
Let \eqref{Integrierbarkeit} be fullfilled. If $\rho_{\cQ}$ is continuous from above at $0$, then under assumption (A) from Section \ref{general abstract minimax}
$$
\sup_{\tau\in\cT_{f}}\inf_{\qr\in\overline{\cQ}}\ex_{\qr}[Y_{\tau}] = \inf_{\qr\in\overline{\cQ}}\sup_{\tau\in\cT_{f}}\ex_{\qr}[Y_{\tau}].
$$
\end{proposition}
\begin{proof}
By assumption
$$
Y_{1/k}\to Y_{0}~\mbox{pointwise for}~k\to\infty.
$$
Then 
$$
\sup_{l\geq k}|Y_{1/l} - Y_{0}|\searrow 0~\mbox{pointwise for}~k\to\infty,
$$
and $\sup_{l\geq k}|Y_{1/l} - Y_{0}|\in L^{1}(\cQ)$ due to \eqref{Integrierbarkeit} along with $\sup_{l\geq k}|Y_{1/l} - Y_{0}|\leq 2\sup_{t\in [0,T]}|Y_{t}|$. Since $\rho_{\cQ}$ is continuous from above at $0$, we may conclude
\begin{eqnarray*}
0
\leq 
\big|\inf_{\qr\in\overline{\cQ}}\ex_{\qr}[Y_{1/k}] - \inf_{\qr\in\overline{\cQ}}\ex_{\qr}[Y_{0}]\big|
&\leq&
\sup_{\qr\in\overline{\cQ}}\ex_{\qr}[|Y_{1/k} - Y_{0}|]\\
&\leq& 
\rho_{\cQ}\big(\sup_{l\geq k}|Y_{1/l} - Y_{0}|\big)\to 0\quad\mbox{for}~k\to\infty.
\end{eqnarray*}
In particular
\begin{equation}
\label{ohne Null 1}
\sup_{\tau\in\cT_{f}}\inf_{\qr\in\overline{\cQ}}\ex_{\qr}[Y_{\tau}] = \sup_{\tau\in\cT_{f}\setminus\{0\}}\inf_{\qr\in\overline{\cQ}}\ex_{\qr}[Y_{\tau}],
\end{equation}
and 
\begin{equation}
\label{ohne Null 2}
\inf_{\qr\in\overline{\cQ}}\sup_{\tau\in\cT_{f}}\ex_{\qr}[Y_{\tau}] = \inf_{\qr\in\overline{\cQ}}\sup_{\tau\in\cT_{f}\setminus\{0\}}\ex_{\qr}[Y_{\tau}].
\end{equation}
We want to apply K\"onig's minimax theorem (cf. \cite[Theorem 4.9]{Koenig1982}) to the mapping
$$
h:\overline{\cQ}\times\cT_{f}\setminus\{0\}\rightarrow\R,~(\qr,\tau)\mapsto \ex_{\qr}[-Y_{\tau}].
$$
For preparation we endow $\overline{\cQ}$ with the topology $\sigma(\overline{\cQ},\cX)$ as defined in subsection \ref{spezielle topological closure}. Then by definition of $\sigma(\overline{\cQ},\cX)$ along with \eqref{Wohldefiniertheit} we may observe
\begin{equation}
\label{continuity condition}
h(\cdot,\tau)~\mbox{is continuous w.r.t.}~\sigma(\overline{\cQ},\cX)\quad\mbox{for}~\tau\in\cT_{f}\setminus\{0\}.
\end{equation}
By convexity of $\overline{\cQ}$ (see \eqref{Vorbereitungen}) we may also observe for $\qr_{1},\qr_{2}\in\cQ, \lambda\in [0,1], \tau\in\cT_{f}$
\begin{equation}
\label{concavity condition}
h(\lambda\qr_{1} + (1-\lambda)\qr_{2},\tau) = \lambda h(\qr_{1},\tau) + (1-\lambda) h(\qr_{2},\tau).
\end{equation}
In view of K\"onig's minimax result along with \eqref{continuity condition}, \eqref{concavity condition} and Lemma \ref{compactness condition} it remains to investigate when the following property is satisfied.
\begin{align}
\label{convexity condition}
\exists~\lambda\in ]0,1[~\forall~\tau_{1},\tau_{2}\in\cT_{f}\setminus\{0\}:
\inf_{\tau\in\cT_{f}\setminus\{0\}}\sup_{\qr\in \overline{\cQ}}\big[h(\qr,\tau) - \lambda h(\qr,\tau_{1}) - (1-\lambda)h(\qr,\tau_{2})\big]\leq 0
\end{align}
By assumption {\rm (A)}, there exists some $\lambda\in ]0,1[$ such that for $\tau_{1},\tau_{2}\in\cT_{f}\setminus\{0\}$ 
\begin{equation}
\label{(AA)}
\inf_{A\in\cF_{\tau_{1}\wedge\tau_{2}}}\rho_{\cQ}((\eins_{A} - \lambda) (Y_{\tau_{2}} - Y_{\tau_{1}}))\leq 0.
\end{equation}
Next, define for arbitrary $\tau_{1},\tau_{2}\in\cT_{f}\setminus\{0\}$ and $A\in \cF_{\tau_{1}\wedge\tau_{2}}$ the map $\tau_{A} \dot= \eins_{A}\tau_{1} + \eins_{\Omega\setminus A}\tau_{2}$. Since $\cF_{\tau_{1}\wedge\tau_{2}}\subseteq\cF_{\tau_{1}}\cap\cF_{\tau_{2}}$, we obtain 
$A\in \cF_{\tau_{i}}$ for $i=1,2$, and
$$
\{\tau_A\leq t\} = \big(A\cap \{\tau_{1}\leq t\}\big) \cup \big(\Omega\setminus A\cap \{\tau_{2}\leq t\}\big)\in\cF_{t}\quad\mbox{for}~t\in [0,T].
$$
In particular $\tau_{A}\in\cT_{f}\setminus\{0\}$ for $A\in\cF_{\tau_{1}\wedge\tau_{2}}$, and thus in view of \eqref{Vorbereitungen} and \eqref{(AA)}
\begin{eqnarray*}
&&
\inf_{\tau\in\cT_{f}\setminus\{0\}}\sup_{\qr\in \overline{\cQ}}\big[h(\qr,\tau) - \lambda h(\qr,\tau_{1}) - (1-\lambda)h(\qr,\tau_{2})\big]\\
&\leq& 
\inf_{A\in\cF_{\tau_{1}\wedge\tau_{2}}}\sup_{\qr\in \overline{\cQ}}\big[h(\qr,\tau_{A}) - \lambda h(\qr,\tau_{1}) - (1-\lambda)h(\qr,\tau_{2})\big]\\
&\stackrel{\eqref{Vorbereitungen}}{=}&
\inf_{A\in\cF_{\tau_{1}\wedge\tau_{2}}}\rho_{\cQ}\big((\eins_{A} - \lambda)(Y_{\tau_{2}} - Y_{\tau_{1}})\big)\stackrel{\eqref{(AA)}}{\leq} 0.
\end{eqnarray*}
 This shows \eqref{convexity condition}, and by K\"onig's minimax theorem we obtain 
 $$
 \inf_{\tau\in\cT_{f}\setminus\{0\}}\sup_{\qr\in\overline{\cQ}}\ex_{\qr}[-Y_{\tau}]
=
\inf_{\tau\in\cT_{f}\setminus\{0\}}\sup_{\qr\in\overline{\cQ}}h(\qr,\tau).
 $$
In view of \eqref{ohne Null 1} along with \eqref{ohne Null 2} this completes the proof of Proposition \ref{Hilfsminimax}. 
\end{proof}
Now, we are ready to show Theorem \ref{generalminimax}.
\medskip

\noindent
\underline{Proof of Theorem \ref{generalminimax}}:\\[0.1cm]
Under the assumptions of Theorem \ref{generalminimax} we may apply Proposition \ref{same optimal value} along with Proposition \ref{same dual optimal value} to obtain
$$
\sup_{\tau\in\cT}\inf_{\qr\in\overline{\cQ}}\ex_{\qr}[Y_{\tau}] = \sup_{\tau\in\cT}\inf_{\qr\in co(\cQ)}\ex_{\qr}[Y_{\tau}]\quad\mbox{and}\quad\inf_{\qr\in\overline{\cQ}}\sup_{\tau\in\cT}\ex_{\qr}[Y_{\tau}] = \inf_{\qr\in co(\cQ)}\sup_{\tau\in\cT}\ex_{\qr}[Y_{\tau}].
$$
Moreover, in view of Lemma \ref{finite range stopping times} along with Proposition \ref{Hilfsminimax} we have
$$
\sup_{\tau\in\cT}\inf_{\qr\in\overline{\cQ}}\ex_{\qr}[Y_{\tau}] 
= 
\sup_{\tau\in\cT_{f}}\inf_{\qr\in\overline{\cQ}}\ex_{\qr}[Y_{\tau}] 
= 
\inf_{\qr\in\overline{\cQ}}\sup_{\tau\in\cT_{f}}\ex_{\qr}[Y_{\tau}]
=
\inf_{\qr\in\overline{\cQ}}\sup_{\tau\in\cT}\ex_{\qr}[Y_{\tau}].
$$
Now, the statement of Theorem \ref{generalminimax} follows immediately.
\hfill$\Box$
 \subsection{Proof of Proposition \ref{Basiskriterium}}
 \label{proof Basiskriterium}
Let the assumptions from the display of Proposition \ref{Basiskriterium} be fulfilled. Fix an arbitrary $\varepsilon > 0$. Observe $|Y_{\tau_{1}} - Y_{\tau_{2}}|\eins_{\{|Y_{\tau_{1}} - Y_{\tau_{2}}| > k\}}\searrow 0$ for $k\to\infty$. Since $\rho_{\cQ}$ is continuous from above at $0$, we may select some $k_{0}\in\N$ such that 
\begin{equation}
\label{Schneidebedingung}
\rho_{\cQ}\left(|Y_{\tau_{1}} - Y_{\tau_{2}}|\eins_{\{|Y_{\tau_{1}} - Y_{\tau_{2}}| > k_{0}\}}\right)\leq  \varepsilon/3.
\end{equation}
The random variable $|Y_{\tau_{1}} - Y_{\tau_{2}}|\eins_{\{|Y_{\tau_{1}} - Y_{\tau_{2}}| \leq k_{0}\}}$ is bounded so that we may find some random variable $X$ on $(\Omega,\cF,\pr)$ with finite range satisfying 
$$
\sup_{\omega\in\Omega}\big|(Y_{\tau_{2}}(\omega) - Y_{\tau_{1}}(\omega))\eins_{\{|Y_{\tau_{1}} - Y_{\tau_{2}}| \leq k_{0}\}}(\omega) - X(\omega)\big|\leq\varepsilon/3
$$
(cf. e.g. \cite[Proposition 22.1]{Koenig1997}). In particular with $\widetilde{Y}~\dot=~Y_{\tau_{2}} - Y_{\tau_{1}}$ 
\begin{equation}
\label{zweite Approximation}
\rho_{\cQ}\big(\big|\widetilde{Y}\eins_{\{|\widetilde{Y}|\leq k_{0}\}} - X\big|\big)
\leq 
\sup_{\omega\in\Omega}|\widetilde{Y}(\omega)\eins_{\{|\widetilde{Y}|\leq k_{0}\}}(\omega) - X(\omega)|
\leq
\varepsilon/3.
\end{equation}
Since $X$ has finite range, there exist pairwise disjoint $B_{1},\dots,B_{r}\in\cF$ and $\lambda_{1},\dots,\lambda_{r}\in\R$ such that $X = \sum_{i=1}^{r}\lambda_{i}\eins_{B_{i}}$. Now, let $(A_{k})_{k\in\N}$ be any sequence in $\cF$. We may observe by assumption that 
any sequence $\big(\mu_{\cQ}(A_{k}\cap B_{i})\big)_{k\in\N}$ is relatively $\|\cdot\|_{\infty}$-compact for $i = 1,\dots,r$ so that there exist a subsequence $(A_{\phi(k)})_{k\in\N}$ and 
$f_{1},\dots,f_{r}\in l^{\infty}(\cQ)$ such that
$$
|\mu_{\cQ}(A_{\varphi(k)}\cap B_{i}) - f_{i}\|_{\infty} \xrightarrow[k \to \infty]{} 0\quad\mbox{for every}~i\in\{1,\dots,r\}.
$$
Then
$$
\sup_{\qr\in\cQ}|\ex_{\qr}[\eins_{A_{\varphi(k)}}\cdot X] - \sum_{i=1}^{r}\lambda_{i} f_{i}(\qr)|
\leq 
\sum_{i=1}^{r}|\lambda_{i}|~\|\mu_{\cQ}(A_{\varphi(k)}\cap B_{i}) - f_{i}\|_{\infty} \xrightarrow[k \to \infty]{} 0.
$$
This means that

\begin{equation}
\label{Basiskompaktheit}
\big\{\big(\ex_{\qr}[\eins_{A}\cdot X]\big)_{\qr\in\cQ}\mid A\in\cF\big\}~\mbox{is relatively}~\|\cdot\|_{\infty}\mbox{-compact}.
\end{equation}
Next, let $L^{1}(\Omega,\cF_{\tau_{1}\wedge\tau_{2}},\pr_{|\cF_{\tau_{1}\wedge\tau_{2}}})$ denote the $L^{1}$-spaces on $(\Omega,\cF_{\tau_{1}\wedge\tau_{2}},\pr_{|\cF_{\tau_{1}\wedge\tau_{2}}})$, whereas we use notation $L^{\infty}(\Omega,\cF_{\tau_{1}\wedge\tau_{2}},\pr_{\cF_{|\tau_{1}\wedge\tau_{2}}})$ for the space of all $\pr_{|\cF_{\tau_{1}\wedge\tau_{2}}}$-essentially bounded random variables. The latter space will be equipped with the weak*-topology $\sigma(L^{\infty}_{\tau_{1}\wedge\tau_{2}},L^{1}_{\tau_{1}\wedge\tau_{2}})$. Since the probability space $(\Omega,\cF_{\tau_{1}\wedge\tau_{2}},\pr_{|\cF_{\tau_{1}\wedge\tau_{2}}})$ is assumed to be atomless, we already know from \cite[Lemma 3]{KingmanRobertson1968}
that $\{\eins_{A}\mid A\in\cF_{\tau_{1}\wedge\tau_{2}}\}$ is a $\sigma(L^{\infty}_{\tau_{1}\wedge\tau_{2}},L^{1}_{\tau_{1}\wedge\tau_{2}})$-dense subset of $\Delta$ consisting of all $Z\in L^{\infty}(\Omega,\cF_{\tau_{1}\wedge\tau_{2}},\pr_{|\cF_{\tau_{1}\wedge\tau_{2}}})$ satisfying 
$0\leq Z\leq 1$ $\pr$-a.s.. In particular we may find a net $(A_{i})_{i\in I}$ such that 
$(\eins_{A_{i}})_{i\in I}$ converges to $1/2$ w.r.t. $\sigma(L^{\infty}_{\tau_{1}\wedge\tau_{2}},L^{1}_{\tau_{1}\wedge\tau_{2}})$. In view of \eqref{Basiskompaktheit}, there is a subnet $(\eins_{A_{i(j)}})_{j\in J}$ such that
$$
\lim_{j }\sup_{\qr\in\cQ}|\ex_{\qr}[\eins_{A_{i(j)}} X] - f(\qr)| = 0\quad\mbox{for some}~f\in l^{\infty}(\cQ).
$$
Notice further that $\ex\big[X \frac{d\qr}{d\pr}~|~\cF_{\tau_{1}\wedge\tau_{2}}\big]$ belongs to $L^{1}(\Omega,\cF_{\tau_{1}\wedge\tau_{2}},\pr_{|\cF_{\tau_{1}\wedge\tau_{2}}})$ for every $\qr\in\cQ$. This implies for any $\qr\in\cQ$
\begin{eqnarray*}
f(\qr) = \lim_{j \to \infty}\ex_{\qr}[\eins_{A_{i(j)}} X] 
= \lim_{j \to \infty}\ex\left[\eins_{A_{i(j)}} X \frac{d\qr}{d\pr}\right] 
&=& 
\lim_{j \to \infty}\ex\left[\eins_{A_{i(j)}} \ex \left[X \frac{d\qr}{d\pr}~\bigg|~\cF_{\tau_{1}\wedge\tau_{2}}\right]\right]\\
&=&
\ex\left[\ex \left[X \frac{d\qr}{d\pr}~\bigg|~\cF_{\tau_{1}\wedge\tau_{2}}\right]/2\right]
=
\ex_{\qr}\big[X/2\big].
\end{eqnarray*}
Hence
\begin{equation}
\label{dritte Approximation}
\sup_{\qr\in\cQ}\big|\ex_{\qr}[\eins_{A_{i(j_{0})}} X] - \ex_{\qr}\big[X/2\big]\big| < \varepsilon/3\quad\mbox{for some}~j_{0}\in J.
\end{equation}
We may observe directly by sublinearity of $\rho_{\cQ}$ along with \eqref{Schneidebedingung}, \eqref{zweite Approximation} and \eqref{dritte Approximation}
\begin{eqnarray*}
&&
\rho_{\cQ}\big((\eins_{A_{i(j_{0})}} - 1/2)\cdot (Y_{\tau_{2}} - Y_{\tau_{1}})\big)\\
&\leq& 
\rho_{\cQ}\big((\eins_{A_{i(j_{0})}} - 1/2)\cdot \tY_{ >~ k_{0}}\big)+ \rho_{\cQ}\big((\eins_{A_{i(j_{0})}} - 1/2)\cdot (\tY_{ \leq~ k_{0}} - X)\big) + 
\rho_{\cQ}\big((\eins_{A_{i(j_{0})}} - 1/2)\cdot X\big)\\
&\leq& 
\rho_{\cQ}\big(|\tY_{>~k_{0}}|) + \rho_{\cQ}\big(|\tY_{\leq~k_{0}} - X|\big)
+ \rho_{\cQ}\big((\eins_{A_{i(j_{0})}} - 1/2)\cdot X\big)
\leq
\varepsilon/3 + \varepsilon/3 + \varepsilon/3 = \varepsilon,
\end{eqnarray*}
where $\tY_{>~ k_{0}}~\dot=~\tY\cdot\eins_{\{|\tY| > k_{0}\}}$ and $\tY_{\leq ~k_{0}}~\dot=~\tY\cdot\eins_{\{|\tY|\leq k_{0}\}}$.
Hence we have shown
$$
\inf_{A\in\cF_{\tau_{1}\wedge\tau_{2}}}\rho_{\cQ}\big((\eins_{A} - 1/2)\cdot (Y_{\tau_{2}} - Y_{\tau_{1}})\big)\leq \varepsilon
$$
which completes the proof by sending $\varepsilon\searrow 0$. 
\hfill$\Box$
\subsection{Proof of Theorem \ref{secondMinimax}}
\label{proof of secondMinimax}
Note first that for $\tau_{1},\tau_{2}\in\cT_{f}\setminus\{0\}$, there is some $t > 0$ such that $\cF_{t}\subseteq\cF_{\tau_{1}\wedge\tau_{2}}$. Therefore by assumption \eqref{nonatomicdominated} the 
probability space $(\Omega,\cF_{\tau_{1}\wedge\tau_{2}},\pr_{|\cF_{\tau_{1}\wedge\tau_{2}}})$ is atomless for $\tau_{1},\tau_{2}$ from $\cT_{f}\setminus\{0\}$. Then in view of Proposition \ref{Basiskriterium} along with Theorem \ref{generalminimax} it remains to show the following auxiliary result.
\begin{lemma}
\label{Stetigkeit von oben}
Let assumption \eqref{Integrierbarkeit} be fulfilled, and let the range of $\mu_{\cQ}$ be 
relatively $\|\cdot\|_{\infty}$-compact. If $\rho_{\cQ}(Y^{*}\eins_{\{Y^{*} > a\}})\to 0$ for $a\to\infty$, then $\rho_{\cQ}$ is continuous from above at $0$.
\end{lemma}
\begin{proof}
Let $(X_{n})_{n\in\N}$ be any nonincreasing sequence in $\cX$ with $X_{n}\searrow 0$ $\pr$-a.s., and let $\varepsilon > 0$. As a member of $\cX$ we may find some $C > 0$ such that
$$
0\leq X_{n}\leq X_{1}\leq C(Y^{*} + 1)~\pr\text{-a.s.}\quad\mbox{for}~n\in\N.
$$
Then for any $n\in\N$ and every $k\in\N$ we may observe by sublinearity of $\rho_{\cQ}$
\begin{align}
\nonumber
0\leq\rho_{\cQ}(X_{n})&\leq \rho_{\cQ}(X_{n}\cdot\eins_{\{Y^{*} \leq k\}})
+ 
\rho_{\cQ}(X_{n}\cdot\eins_{\{Y^{*} > k\}})\\
&\leq
\nonumber
\rho_{\cQ}(X_{n}\cdot\eins_{\{Y^{*} \leq k\}}) + \rho_{\cQ}(C(Y^{*} + 1)\cdot\eins_{\{Y^{*} > k\}})\\
&\leq
\label{schneiden}
\rho_{\cQ}(X_{n}\cdot\eins_{\{Y^{*} \leq k\}}) + 2 C\rho_{\cQ}(Y^{*}\cdot\eins_{\{Y^{*} > k\}}).
\end{align}
Next, observe that $\big(X_{n}\cdot\eins_{\{Y^{*} \leq k\}}\big)_{n\in\N}$ is uniformly bounded by some constant say $C_{k}$ for any $k\in\N$. 
Then with 
$\overline{X}_{k,n} \dot= X_{n}\cdot\eins_{\{Y^{*} \leq k\}}$ we obtain for $k,n\in\N$
\begin{equation}
\label{Horizontaldarstellung}
0\leq
\rho_{\cQ}(\overline{X}_{k,n}) 
= 
\sup_{\qr\in\cQ}\int_{0}^{\infty}\qr\big(\big\{\overline{X}_{k,n} > x\big\}\big)~dx 
\leq 
\int_{0}^{C_{k}}\sup_{\qr\in\cQ}\qr\big(\big\{\overline{X}_{k,n}> x\big\}\big)~dx.
\end{equation}
Now fix $k\in\N$ and $x\in ]0, C_{k}[$. Since the range of $\mu_{\cQ}$ is assumed to be relatively compact w.r.t. $\|\cdot\|_{\infty}$, we may find for any subsequence 
$\big(\mu_{\cQ}\big(\big\{\overline{X}_{k,(i(n))}> x\big\}\big)\big)_{n\in\N}$ a further subsequence $\big(\mu_{\cQ}\big(\big\{\overline{X}_{k,(j(i(n)))}> x\big\}\big)\big)_{n\in\N}$ such that
$$
\lim_{n\to\infty}\sup_{\qr\in\cQ}|\mu_{\cQ}\big(\big\{\overline{X}_{k,j(i(n))} > x\big\}\big)(\qr) - f_{k}(\qr)|\to 0
$$
for some $f_{k}\in l^{\infty}(\cQ)$. Furthermore, $\qr\big(\big\{\overline{X}_{k,j(i(n))}> x\big\}\big) \searrow 0$ for $n\to\infty$ if $\qr\in\cQ$.
%
This implies $f_{k}\equiv 0$, and thus $\sup_{\qr\in\cQ}\qr\big(\big\{\overline{X}_{k,n}> x\big\}\big)\searrow 0$ for $n\to\infty$. Hence in view of \eqref{Horizontaldarstellung}, the application of the dominated convergence theorem yields $\rho_{\cQ}(X_{n}\cdot\eins_{\{Y^{*} \leq k\}})\to 0$ for $n\to\infty$ with $k\in\N$ fixed.
Then by \eqref{schneiden}
\begin{eqnarray*}
0\leq\limsup_{n\to\infty}\rho_{\cQ}(X_{n})\leq 2 C\rho_{\cQ}(Y^{*}\cdot\eins_{\{Y^{*} > k\}})\quad\mbox{for}~k\in\N.
\end{eqnarray*}
Finally, by assumption, $C\rho_{\cQ}(Y^{*}\cdot\eins_{\{Y^{*} > k\}})\to 0$ for $k\to\infty$ so that 
$$
0\leq\limsup_{n\to\infty}\rho_{\cQ}(X_{n})\leq 0.
$$
This completes the proof.

\end{proof}
As discussed just before Lemma \ref{Stetigkeit von oben}, the statement of Theorem \ref{secondMinimax} follows immediately by combining Proposition \ref{Basiskriterium} and Theorem \ref{generalminimax} with Lemma \ref{Stetigkeit von oben}.
\hfill$\Box$
\subsection{Proof of Theorem \ref{point compact class}}
\label{Theorem point compact class}

Since $d$ is totally bounded the completion $(\check{\cQ},\check{d})$ of $(\cQ,d)$ is compact (see \cite[Sec. 9.2, Problem 2]{Wilansky1970}). Since 
 $(d\qr/d\pr)_{\qr\in\cQ}$ has $\pr$-almost surely $d$-uniformly continuous paths, we may find some $A\in\cF$ with $\pr(A) = 1$ such that we may define a nonnegative stochastic process $(Z_{\check{\qr}})_{\check{\qr}\in\check{\cQ}}$ such that 
 $Z~\dot=~Z_{\cdot}(\omega)$ is continuous for $\omega\in A$ and $Z_{\qr} = (d\qr/d\pr)$ holds for $\qr\in\cQ$ (see \cite[Theorem 11.3.4]{Wilansky1970}). 

 \medskip

 Now let $(\qr_{n})_{n\in\N}$ be any sequence in $\cQ$. By compactness we may select a subsequence $(\qr_{i(n)})_{n\in\N}$ which converges to some $\check{\qr}\in\check{\cQ}$ w.r.t. $\check{d}$ (see \cite[Theorem 7.2.1]{Wilansky1970}). Since the process $Z$ has $\check{d}$-continuous paths on $A$, we obtain
 $$
 \frac{d\qr_{i(n)}}{d\pr}(\omega) = Z_{\qr_{i(n)}}(\omega)\xrightarrow[n\to\infty]{} Z_{\check{\qr}}(\omega)\quad\mbox{for all}~\omega\in A.
 $$
 Moreover, by assumption, $(d\qr_{i(n)}/d\pr)_{n\in\N}$ is dominated by some $\pr$-integrable random variable $U$. Then the application of the dominated convergence theorem yields
 $$
 \ex\left[\left|\frac{d\qr_{i(n)}}{d\pr} - Z_{\check{\qr}}\right|\right]\xrightarrow[n\to\infty]{} 0.
 $$
 Thus we have shown that $\cQ$ is relatively compact w.r.t. topology of total variation as defined in Remark \ref{total variation}, and Theorem \ref{point compact class} may be concluded by combining Remark \ref{total variation} with Theorem \ref{secondMinimax}.
 \hfill$\Box$
 
\subsection{Proof of Proposition \ref{prop_suffLLN3}}
\label{LLN3}
Fix $t \in [0,T]$. By assumption $(X^{\theta})_{\theta\in\Theta}$ is a nearly sub-Gaussian random field in the sense of the Appendix. Then}
by Proposition \ref{continuity sub Gaussian}, we may fix some separable version $(\widehat{X}_t^\theta)_{\theta \in \Theta}$ of 
$(X_t^\theta)_{\theta \in \Theta}$. 
It is also assumed that there is some $\overline{\theta} \in \Theta$, such that $\ex[\exp(2 \widehat{X}^{\overline{\theta}}_t)] = \ex[\exp(2 X^{\overline{\theta}}_t)] < \infty$ holds.
In addition, by Proposition \ref{continuity sub Gaussian} again, we may also find a nonnegative random variable $U^{\overline{\theta}}_t$ as well as some $A_t\in\cF$ with $\pr(A_t) = 1$, such that 
\begin{eqnarray}
&&
\label{Condition 1}
\ex[\exp(p U^{\overline{\theta}}_t)] < \infty\quad\mbox{for every}~p\in ]0,\infty[\\
&&
\label{Condition 2}
\sup_{\theta\in\Theta}\exp\left(\widehat{X}^{\theta}_t(\omega)\right)\leq\exp\left(U^{\overline{\theta}}_t(\omega)\right) \exp\left(\widehat{X}^{\overline{\theta}}_t(\omega)\right)\quad\mbox{for}~\omega\in A_t.
\end{eqnarray}
By assumption and since $A_t \in \cF_t$, for every $\theta\in\Theta$ 
$$
M^{\theta}_t\doteq\exp\left(\widehat{X}^{\theta}_t - [X^\theta]_t/2\right)\eins_{A_t}
$$
defines a Radon-Nikodym derivative of 
${\qr_{\theta}}_{|\cF_t}$. 
Then due to the nonnegativity of the process $([X^\theta]_t)_{\theta \in \Theta}$ the application of \eqref{Condition 2} yields
\begin{equation*}
\label{Abschaetzung}
\sup_{\theta\in\Theta}M^{\theta}_t\leq \exp\left(U^{\overline{\theta}}_t\right)\exp\left(\widehat{X}^{\overline{\theta}}_t\right)~\quad\mbox{pointwise}.
\end{equation*}
By \eqref{Condition 1} along with the assumptions on $X^{\overline{\theta}}_{t}$, the random variables $\exp(2 U^{\overline{\theta}}_t)$ 
and $\exp(\widehat{X}^{\overline{\theta}}_t)$ are square integrable.
Hence by Cauchy-Schwarz inequality 
$\exp(U^{\overline{\theta}}_t)\exp(\widehat{X}^{\overline{\theta}}_t)$ is integrable so that $(M^{\theta}_t)_{\theta\in\Theta}$ is dominated by some $\pr$-integrable random variable. 
Thus by Theorem \ref{point compact class} it remains to show that $(M^{\theta}_t)_{\theta\in\Theta}$ has $d_{\Theta}$-uniformly continuous paths.
For $\theta, \vartheta\in\Theta$ we may conclude from \eqref{Condition 2} and the nonnegativity of the process $([X^\theta]_t)_{\theta \in \Theta}$
\begin{eqnarray}
|M^{\theta}_t - M^{\vartheta}_t|
&\leq&
\nonumber
 \left(
\exp\left(\widehat{X}^{\theta}_t\right) 
+ \exp\left(\widehat{X}^{\vartheta}_t\right)\right) 
\left(\Big|\widehat{X}^{\theta}_t 
- \widehat{X}^{\vartheta}_t\Big|  
+ \Big|[X^\theta]_t/2 
- [X^\vartheta]_t/2\Big|\right)\eins_{A_t}\\
&\leq& 
\nonumber
\sup_{\theta\in\Theta}
\exp\left(\widehat{X}^{\theta}_t\right) 
\left(\Big|\widehat{X}^{\theta}_t 
- \widehat{X}^{\vartheta}_t\Big|  
+ \Big|[X^\theta]_t/2 
- [X^\vartheta]_t/2\Big|\right)\eins_{A_t}\\
&\stackrel{\eqref{Condition 2}}{\leq}& 
\label{gleichmassige Pfade}
\exp\left(U^{\overline{\theta}}_t\right) 
\exp\left(\widehat{X}^{\overline{\theta}}_t\right) 
\left(\Big|\widehat{X}^{\theta}_t 
- \widehat{X}^{\vartheta}_t\Big|  
+ \Big|[X^\theta]_t/2 
- [X^\vartheta]_t/2\Big|\right)\eins_{A_t}.
\end{eqnarray}
In view of Proposition \ref{continuity sub Gaussian} the process $(\widehat{X}^\theta_t)_{\theta\in\Theta}$  has $d_{\Theta}$-uniformly continuous paths and $([X^\theta]_t)_{\theta \in \Theta}$ satisfies this property by assumption. 
Thus by \eqref{gleichmassige Pfade}, $(M^{\theta}_t)_{\theta\in\Theta}$ has $d_{\Theta}$-uniformly continuous paths.
\hfill$\Box$
\subsection{Proof of Example \ref{stochastic integrals}}
\label{proof of stochastic integrals}
Firstly, each $X^{\psi}$ is a centered martingale. Secondly, by time change we may construct 
an enlargement $\oOFFP$ of the filtered probability space $\OFFP$ with $\overline{\Omega} = \Omega\times\widetilde{\Omega}$ for some set $\widetilde{\Omega}$ such that for  every fixed pair $\psi, \phi \in \Psi$ there exists a Brownian motion $\overline{Z}^{\psi, \phi}$ with $(\overline{Z}^{\psi, \phi}_t)_{0 \leq t \leq T}$ being adapted to $(\cF_t)_{0 \leq t \leq T}$, and for every $t \in[0,T]$ it holds
\begin{align*}
\overline{X}_{t}^\psi - \overline{X}_{t}^\phi
= \overline{Z}^{\psi, \phi}_{\int_0^{t} (\psi-\phi)^2(u, V_u)~du}
\end{align*}
(see e.g. \cite[proof of Theorem V.1.7]{RevuzYor1991}). 
Here we set for each $\overline{\omega} = (\omega, \tilde{\omega})\in \overline{\Omega}$
\begin{align*}
\overline{X}_{t}^\psi(\overline{\omega}) 
- \overline{X}_{t}^\phi(\overline{\omega})
= \overline{X}_{t}^\psi(\omega, \tilde{\omega}) 
- \overline{X}_{t}^\phi(\omega, \tilde{\omega})~
\dot=~X_{t}^\psi(\omega) - X_{t}^\phi(\omega).
\end{align*}
Then for fixed $\lambda > 0$, $t \in [0,T]$ and $\psi, \phi \in \Psi$ we obtain 
\begin{align*}
\mathbb{E}[ \exp(\lambda(X_{t}^\psi - X_{t}^\phi))]
= \mathbb{E}_{\overline{\pr}}\left[ 
\exp\left(\lambda 
\overline{Z}^{\psi, \phi}_{\int_0^{t} (\psi - \phi)^2(u, V_u)~du}
\right)\right]
\leq \mathbb{E}_{\overline{\pr}}\left[ 
\exp\left(\lambda 
\max_{0 \leq s \leq d(\psi, \phi)^2} \overline{Z}^{\psi, \phi}_s
\right)\right].
\end{align*}
Now we derive by the reflection principle for Brownian motion
\begin{align*}
\mathbb{E}_{\overline{\pr}}\left[ 
\exp\left(\lambda 
\max_{0 \leq s \leq d(\psi, \phi)^2} \overline{Z}^{\psi, \phi}_s
\right)\right] 
\leq 2 \mathbb{E}_{\overline{\pr}}\left[ 
\exp \left( \lambda \overline{Z}^{\psi, \phi}_{d(\psi, \phi)^2} 
\right)\right] 
& = 2 \exp \left( \frac{\lambda^2 d(\psi, \phi)^2}{2} \right).
\end{align*}
Hence $(X_t^{\psi},\psi\in\Psi)$ is a nearly sub-Gaussian family  of  local martingales with $C = 2$. 
\hfill$\Box$

\subsection{Proof of Proposition \ref{bereits Zeitkonsistenz}}
\label{Zeitkonsistenz}
Let $\widehat{\cQ}$ and $\widehat{\cQ}^{e}$ be defined as in Proposition \ref{bereits Zeitkonsistenz}. Furthermore, let $L^{p}\OFP$ denote the classical $L^{p}$-space on $\OFP$ for $p\in [1,\infty]$. We shall need the following auxiliary result for preparation.
\begin{lemma}
\label{Kompaktheit konvex Huelle}
The set $\mathbb{F}_{\widehat{\cQ}}~\dot=~ \{d\qr/d\pr~|~\qr\in\widehat{\cQ}\}$ is closed w.r.t. the $L^{1}-$norm. It is even compact w.r.t. the $L^{1}$-norm if 
$\cQ$ is relatively compact w.r.t. the topology of total variation. In this case 
$
\mathbb{F}_{\widehat{\cQ}^{e}}~\dot=~\{d\qr/d\pr~|~\qr\in\widehat{\cQ}^{e}\}
$
is relatively compact w.r.t. the $L^{1}$-norm.
\end{lemma}
\begin{proof}
The set $\mathbb{F}_{\widehat{\cQ}}$ is obviously convex, and it is also known to be the topological closure of the convex hull $co(\mathbb{F}_{\cQ})$ of
$
\mathbb{F}_{\cQ}
$
w.r.t. the weak topology on $L^{1}(\Omega,\cF,\pr)$ (see \cite[Theorem 1.4]{Koenig2001}). Thus by convexity, $\mathbb{F}_{\widehat{\cQ}}$ is also the closed convex hull of $\mathbb{F}_{\cQ}$ w.r.t. the $L^{1}$-norm topology. Moreover, if $\cQ$ is relatively compact w.r.t. the topology of total variation, the set $\mathbb{F}_{\cQ}$ is relatively $L^{1}$-norm compact so that its $L^{1}$-norm closed convex hull $\mathbb{F}_{\widehat{\cQ}}$ is $L^{1}$-norm compact (see e.g. \cite[Theorem 5.35]{AliprantisBorder2006}). This completes the proof because $\mathbb{F}_{\widehat{\cQ}^{e}} \subseteq\mathbb{F}_{\widehat{\cQ}}$.
\end{proof}

\bigskip

\noindent
\underline{Proof of Proposition \ref{bereits Zeitkonsistenz}}:\\[0.1cm]


The implication $(5)\Rightarrow (4)$ is already known (see \cite[Lemma 5.3]{Trevino2008}, and \cite[Lemma 6.48]{FoellmerSchied2011} for the time-discrete case). Concerning the implication $(1)\Rightarrow (2)$ let $\tau\in\cT$ and $X\in L^{\infty}\OFPT$. Then 
$Z~\dot=~\essinf_{\qr\in\cQ}~\ex_{\qr}[X|\cF_{\tau}]\in L^{\infty}\OFPT$ and is $\cF_{\tau}$-measurable so that 
$$
\essinf_{\qr\in\cQ}~\ex_{\qr}[X|\cF_{\tau}] = \essinf_{\qr\in\cQ}~\ex_{\qr}[Z|\cF_{\tau}].
$$
Then by time-consistency (statement (1))
$$
\inf_{\qr\in\cQ}~\ex_{\qr}[X] = \inf_{\qr\in\cQ}~\ex_{\qr}[Z]
$$
which shows (2).

Next, want to show that (3) may be concluded from (2). Firstly, (2) obviously implies
$$
\inf_{\qr\in\cQ}\ex_{\qr}[X] = \inf_{\qr\in\cQ}~\ex_{\qr}\big[\essinf_{\qr\in\cQ}\ex_{\qr}[X|\cF_{\tau}]\big]\quad\mbox{for}~\tau\in\cT~\mbox{and}~X\in L^{\infty}\OFPT,
$$
which may be rewritten by
\begin{equation}
\label{recursiveness}
\rho_{0}(X) = \rho_{0}(-\rho_{\tau}(X))\quad\mbox{for}~\tau\in\cT~\mbox{and}~X\in L^{\infty}\OFPT,
\end{equation}
where
$$
\rho_{s}(X)\doteq \esssup_{\qr\in\cQ}~\ex_{\qr}[-X|\cF_{s}]\quad\mbox{for}~X\in L^{\infty}\OFP\quad(s\in\{0,\tau\}, \tau \in \cT).
$$
Since each member of $\cQ$ is equivalent to $\pr$ we may observe from \eqref{recursiveness} that for every $\tau\in\cT \setminus \{0\}$ the functions $\rho_{0},\rho_{\tau}$ fulfill the assumptions and statement (a) from Theorem 11.22 in \cite{FoellmerSchied2011}. Then in the proof of this theorem it is shown 
$$
\rho_{\tau}(X) = \esssup_{\qr\in\widehat{\cQ}^{e}}~\ex_{\qr}[-X|\cF_{\tau}]\quad\mbox{for}~X\in L^{\infty}\OFPT
$$
so that
\begin{equation}
\label{zeitkonsistente Darstellung}
\essinf_{\qr\in\cQ}~\ex_{\qr}[X|\cF_{\tau}] = \essinf_{\qr\in\widehat{\cQ}^{e}}~\ex_{\qr}[X|\cF_{\tau}]\quad\mbox{for}~\tau\in\cT~\mbox{and}~X\in L^{\infty}\OFPT.
\end{equation}
In order to verify statement (3) it is left to show that $\widehat{\cQ}^{e}$ is stable under the pasting. So let $\qr_{1},\qr_{2}\in\widehat{\cQ}^{e}, \tau\in\cT$, and let $\overline{\qr}$ denote the pasting of $\qr_{1}, \qr_{2}$ in $\tau$. Then for any $\pr$-essentially bounded random variable $X$
\begin{eqnarray*}
\ex_{\overline{\qr}}[X] 
= 
\ex_{\qr_{1}}\big[\ex_{\qr_{2}}[X|\cF_{\tau}]\big] 
\geq 
\ex_{\qr_{1}}\big[\essinf_{\qr\in\widehat{\cQ}^{e}} \ex_{\qr}[X|\cF_{\tau}]\big]
&\stackrel{\eqref{zeitkonsistente Darstellung}}{=}& 
\ex_{\qr_{1}}\big[\essinf_{\qr\in\cQ} \ex_{\qr}[X|\cF_{\tau}]\big]\\
&\geq& 
\inf_{\qr\in \widehat{\cQ}^{e}}\ex_{\qr}\big[\essinf_{\qr\in\cQ} \ex_{\qr}[X|\cF_{\tau}]\big]\\
&\stackrel{\eqref{zeitkonsistente Darstellung}}{=}&
\inf_{\qr\in \cQ}\ex_{\qr}\big[\essinf_{\qr\in\cQ} \ex_{\qr}[X|\cF_{\tau}]\big].
\end{eqnarray*}
Hence $\ex_{\overline{\qr}}[X]\geq\inf_{\qr\in\cQ}\ex_{\qr}[X]$ holds due to statement (2). Therefore $\overline{\qr}$ belongs to $\widehat{\cQ}$, and thus also to $\widehat{\cQ}^{e}$.
\medskip

Let us now turn over to the implication $(3)\Rightarrow (1)$. So let us assume that (3) is valid and let $\overline{X}, X\in L^{\infty}\OFPT$ as well as $\sigma, \tau\in\cT$ with $\sigma\leq\tau$ such that
$$
\essinf_{\qr\in\cQ}\ex_{\qr}[\overline{X}|\cF_{\tau}]\leq \essinf_{\qr\in\cQ}\ex_{\qr}[X|\cF_{\tau}].
$$
In view of \eqref{zeitkonsistente Darstellung} this means
\begin{equation}
\label{Zeitkonsistenz-Bedingung}
\essinf_{\qr\in\widehat{\cQ}^{e}}\ex_{\qr}[\overline{X}|\cF_{\tau}]\leq \essinf_{\qr\in\widehat{\cQ}^{e}}\ex_{\qr}[X|\cF_{\tau}].
\end{equation}
Let $\varepsilon > 0$ with $|\overline{X}|\leq\varepsilon$ $\pr$-a.s., and define the uniformly bounded, nonnegative càdlàg-process $H~\dot=~(H_{t})_{0 \leq t \leq T}$ via 
$H_{t}~\dot=~\eins_{\{T\}}(t) (\varepsilon - \overline{X})$. Since $\widehat{\cQ}^{e}$ is stable under pasting, the application of Lemma 4.17 in \cite{Trevino2008} to $H$ yields
$$
\esssup_{\qr\in\widehat{\cQ}^{e}}\ex_{\qr}\big[\varepsilon - \overline{X}~|~\cF_{\sigma}~\big] 
= 
\esssup_{\qr\in\widehat{\cQ}^{e}}\ex_{\qr}\big[\esssup_{\qr\in\widehat{\cQ}^{e}}\ex_{\qr}[\varepsilon - \overline{X}~|~\cF_{\tau}]~|~\cF_{\sigma}~\big]. 
$$
In particular, we obtain
$$
\essinf_{\qr\in\widehat{\cQ}^{e}}\ex_{\qr}[\overline{X}|\cF_{\sigma}] 
=\essinf_{\qr\in\widehat{\cQ}^{e}}\ex_{\qr}\big[\essinf_{\qr\in\widehat{\cQ}^{e}}\ex_{\qr}[\overline{X}|\cF_{\tau}]~\big|~\cF_{\sigma}\big].
$$
Then in view of \eqref{zeitkonsistente Darstellung} along with \eqref{Zeitkonsistenz-Bedingung} this implies
$$
\essinf_{\qr\in\cQ}\ex_{\qr}[\overline{X}|\cF_{\sigma}]\leq \essinf_{\qr\in\cQ}\ex_{\qr}[X|\cF_{\sigma}].
$$
\medskip

Concerning implication $(4)\Rightarrow (2)$ let 
$X\in L^{\infty}\OFPT$. There is some $C > 0$ such that 
$X + C\geq 1$ $\pr$-a.s.. Then $Z_{t}~\dot=~\eins_{\{T\}}(t)\cdot (X + C)$ defines 
a uniformly bounded, nonnegative adapted càdlàg process $Z = (Z_{t})_{0 \leq t \leq T}$ from 
$\cS(\cQ)$. 
Furthermore, let us fix $\overline{\qr}\in\cQ$ and $\tau\in\cT$. By statement (4) we shall find some sequence $(\qr^{k})_{k\in\N}$ in $\cQ$ whose members coincide with $\overline{\qr}$ on $\cF_{\tau}$ such that 
$$
\ex_{\qr_{k}}[X + C|\cF_{\tau}] = \esssup_{\sigma\in\cT,\sigma\geq\tau}~\ex_{\qr_{k}}[Z_{\sigma}|\cF_{\tau}] \xrightarrow[k\to\infty]{}  \essinf_{\qr\in\cQ}~\esssup_{\sigma\in\cT,\sigma\geq\tau}~\ex_{\qr}[Z_{\sigma}|\cF_{\tau}]\quad\pr\mbox{-a.s.}.
$$
Since in addition $\esssup_{\sigma\in\cT,\sigma\geq\tau}~\ex_{\qr}[Z_{\sigma}|\cF_{\tau}] = \ex_{\qr}[X + C|\cF_{\tau}]$ holds for every $\qr\in\cQ$, we obtain by dominated convergence theorem
\begin{eqnarray*}
\ex_{\overline{\qr}}[\essinf_{\qr\in\cQ}~\ex_{\qr}[X + C|\cF_{\tau}]] 
=
\lim_{k\to\infty}\ex_{\overline{\qr}}[\ex_{\qr_{k}}[X + C|\cF_{\tau}]]
&=& 
\lim_{k\to\infty}\ex_{\qr^{k}}[\ex_{\qr^{k}}[X + C|\cF_{\tau}]]\\
&=&
\lim_{k\to\infty}\ex_{\qr^{k}}[X + C]\geq\inf_{\qr\in\cQ}\ex_{\qr}[X + C].
\end{eqnarray*}
Here for the second equality we have invoked that $\qr^{k}_{|\cF_{\tau}} = \overline{\qr}_{|\cF_{\tau}}$ holds for every $k\in\N$. Then statement (2) is obvious, and the proof is complete.
\hfill$\Box$

\subsection{Proof of Remark \ref{total variation}}
\label{compact-anti-zeitkonsistent}

If the sets $\widehat{\cQ}$ and $\widehat{\cQ}^{e}$ are defined as in Proposition \ref{bereits Zeitkonsistenz}, then in view of Proposition \ref{bereits Zeitkonsistenz} it remains to show that under the assumptions of Remark \ref{total variation} the 
set $\widehat{\cQ}^{e}$
is not stable under pasting w.r.t. $\OFFP$. Since the set of all probability measures on $\cF$ which are equivalent to $\pr$ is stable under pasting w.r.t. $\OFFP$, enclosing $\widehat{\cQ}^{e}$, we may define the minimal set $\widehat{\cQ}^{\rm st}$ which is stable under pasting and is a superset of $\widehat{\cQ}^{e}$. We want to show 
that $\widehat{\cQ}^{e}$ is a proper subset of $\widehat{\cQ}^{\rm st}$ within the setting of Remark \ref{total variation}. 
The argumentation will be based on the following 
observation.


\begin{lemma}
\label{anti-stabil}
Let $L^{p}\tOFP$ denote the classical $L^{p}$-space on a probability space $\tOFP$ for $p\in [0,\infty]$ and let $(A_{n})_{n\in\N}$ be a sequence in $\overline{\cF}$ satisfying
\begin{align}
\label{eq:23052017a1}
\lim_{n\to\infty}\ex[\eins_{A_{n}}\cdot Z] = \frac{1}{2}\cdot\ex[Z]\quad\mbox{for every}~Z\in L^{1}\tOFP.
\end{align}
Then for any $Z\in L^{1}\tOFP\setminus\{0\}$, the sequence $(\eins_{A_{n}}\cdot Z)_{n\in\N}$ does not have any accumulation point in $L^{1}\tOFP$ w.r.t. the $L^{1}$-norm.
\end{lemma} 
\begin{proof}
Assume that there is some $Z\in L^{1}\tOFP\setminus\{0\}$
such that the sequence $(\eins_{A_n}Z)_{n\in\N}$
has an accumulation point $X\in L^{1}\tOFP$
w.r.t. the $L^1$-norm.
By passing to a subsequence, we can assume that
\begin{equation}
\label{eq:23052017a1.5}
\ex[|\eins_{A_n}Z - X|]\to 0\quad\mbox{for}~n\to\infty.
\end{equation}
Therefore by H\"older's inequality,
\begin{equation}
\label{eq:23052017a2}
\lim_{n\to\infty}\ex[|\eins_{A_n}ZW -XW|] = 0
\quad\text{for any }
W\in L^\infty\tOFP.
\end{equation}
Applying \eqref{eq:23052017a1} and~\eqref{eq:23052017a2},
we get
$\frac12\ex[ZW]=\ex[XW]$ for any $W\in L^\infty\tOFP$,
that is,
\begin{equation}
\label{eq:23052017a3}
\ex\left[\left(\frac12 Z-X\right)W\right]=0
\quad\text{for any }W\in L^\infty\tOFP.
\end{equation}
Substituting
$W\doteq\left[\left(\frac12 Z-X\right)\wedge1\right]\vee-1$
into~\eqref{eq:23052017a3}, we arrive at $X=\frac12 Z$,
hence, by~\eqref{eq:23052017a1.5},
$$
\frac12\ex[|Z|]
=\ex\left[\left|\left(\eins_{A_n}-\frac12\right)Z\right|\right]
\xrightarrow[n\to\infty]{}0.
$$
This contradicts $\overline{\pr}(\{Z\not= 0\}) > 0$ and completes the
proof.
\end{proof}
\bigskip

\noindent
\underline{Proof of Remark \ref{total variation}}:\\


Let us fix different $\qr_{1},\qr_{2}\in\cQ\subseteq\widehat{\cQ}^{e}$. Since $\widehat{\cQ}^{\rm st}$ is stable under pasting we may define for every $\tau\in\cT$ by
$$
\frac{d\qr^{\tau}}{d\pr}~\dot=~\frac{\ex\left[\frac{d\qr_{1}}{d\pr}~\big|~\cF_{\tau}~\right]}{\ex\left[\frac{d\qr_{2}}{d\pr}~\big|~\cF_{\tau}~\right]}~\frac{d\qr_{2}}{d\pr}
$$
a Radon-Nikodym derivative w.r.t. $\pr$ of some probability measure $\qr^{\tau}\in\widehat{\cQ}^{\rm st}$. In particular, $\qr^{0} = \qr_{2}$ and $\qr^{T} = \qr_{1}$, and using the càdlàg-modifications of the density processes
$$
\left(\ex\left[\frac{d\qr_{i}}{d\pr}~\bigg|~\cF_{t}\right]\right)_{0 \leq t \leq T}\quad(i=1,2)
$$
we derive
$$
\frac{d\qr^{t}}{d\pr}\to \frac{d\qr_{2}}{d\pr}\quad\mbox{for}~t\searrow 0.
$$
Therefore, we may find some $t_{0}\in ]0,T[$ such that $\qr^{t_{0}}\not=\qr^{T}$. Since 
by assumption $(\Omega,\cF_{t_{0}},\pr_{|\cF_{t_{0}}})$ is atomless with 
$L^{1}(\Omega,\cF_{t_{0}},\pr_{|\cF_{t_{0}}})$ being weakly separable we may draw on 
\cite[Lemma 3]{KingmanRobertson1968} (or \cite[Corollary C.4]{BelomestnyKraetschmer2016}) along with \cite[Lemma C.1 and Proposition B.1]{BelomestnyKraetschmer2016} to find some sequence $(A_{n})_{n\in\N}$ in $\cF_{t_{0}}$ such that
$$
\lim_{n\to\infty}\ex[\eins_{A_{n}}\cdot Z] = \frac{1}{2}\cdot\ex[Z]\quad\mbox{for every}~\pr_{|\cF_{t_{0}}}\mbox{-integrable random variable}~Z.
$$
In particular 
\begin{equation}
\label{erste Bedingung anti-Zeitkonsistenz}
\lim_{n\to\infty}\ex[\eins_{A_{n}}\cdot Z] = \lim_{n\to\infty}\ex\big[\eins_{A_{n}}\cdot \ex[Z~|~\cF_{t_{0}}]\big] = \frac{1}{2}\cdot\ex\big[\ex[Z~|~\cF_{t_{0}}]\big] = \frac{1}{2}\cdot\ex[Z]
\end{equation}
holds for every $Z\in L^{1}(\Omega,\cF,\pr)$. 
\medskip

Moreover, $\tau_{n}~\dot=~t_{0}\cdot\eins_{A_{n}} + T\cdot\eins_{\Omega\setminus A_{n}}$ defines a sequence $(\tau_{n})_{n\in\N}$ in $\cT$ which induces the sequence $(\qr^{\tau_{n}})_{n\in\N}$ in $\widehat{\cQ}^{\rm st}$ whose Radon-Nikodym derivatives w.r.t. $\pr$ satisfy
\begin{align*}
\frac{d\qr^{\tau_{n}}}{d\pr} 
= 
\eins_{A_{n}}\cdot \frac{d\qr^{t_{0}}}{d\pr} + 
\eins_{\Omega\setminus A_{n}}\cdot \frac{d\qr^{T}}{d\pr}
= 
\eins_{A_{n}}\cdot\left(\frac{d\qr^{t_{0}}}{d\pr} - \frac{d\qr^{T}}{d\pr}\right) + 
\frac{d\qr^{T}}{d\pr}.
\end{align*}
By choice $d\qr^{t_{0}}/d\pr - d\qr^{T}/d\pr\in L^{1}(\Omega,\cF,\pr_{|\cF})\setminus\{0\}$. So in view of Lemma \ref{anti-stabil} along with \eqref{erste Bedingung anti-Zeitkonsistenz} we may observe that the sequence
$$
\left(\eins_{A_{n}}\cdot\left(\frac{d\qr^{t_{0}}}{d\pr} - \frac{d\qr^{T}}{d\pr}\right)\right)_{n\in\N}
$$
does not have any accumulation point in $L^{1}(\Omega,\cF,\pr_{|\cF})$ w.r.t. the  $L^{1}$-norm, and thus the sequence $(d\qr^{\tau_{n}}/d\pr)_{n\in\N}$ has also no accumulation point. Hence we have found a sequence in 
$
\mathbb{F}_{\widehat{\cQ}^{\rm st}}~\dot=~\{d\qr/d\pr~|~\qr\in\widehat{\cQ}^{\rm st}\}
$
without any accumulation point w.r.t. the $L^{1}$-norm. This means that $\mathbb{F}_{\widehat{\cQ}^{\rm st}}$ is not relatively compact w.r.t. the $L^{1}$-norm. However, the set $\mathbb{F}_{\widehat{\cQ}^{e}}$ from Lemma \ref{Kompaktheit konvex Huelle} has been shown there to be relatively compact w.r.t. the $L^{1}$-norm. Hence $\mathbb{F}_{\widehat{\cQ}^{e}}\not=\mathbb{F}_{\widehat{\cQ}^{\rm st}}$, and thus $\widehat{\cQ}^{e}$ is  a proper subset of $\widehat{\cQ}^{\rm st}$. So by construction of $\widehat{\cQ}^{\rm st}$, the set $\widehat{\cQ}^{e}$ is not stable under pasting w.r.t. $\OFFP$ so that $\cQ$ is not time-consistent w.r.t. $\OFFP$ due to Proposition \ref{bereits Zeitkonsistenz}. The proof of Remark \ref{total variation} is complete. 
\hfill$\Box$

\appendix
\section{Appendix: Paths of nearly sub-Gaussian random fields}
\label{Criterion for LLN}

Let $(\Theta,d)$ be some totally bounded semimetric space with diameter $\Delta$. For $\delta, \varepsilon > 0$ the symbols  ${\cal D}(\delta,d)$ and $N(\Theta,d;\varepsilon)$ are used in an analogous manner as the notations ${\cal D}(\delta,d_{\Theta})$ and $N(\Theta,d_{\Theta};\varepsilon)$ from  Section \ref{applications to parameterized families}. We shall call 
a centered stochastic process $(X^{\theta})_{\theta\in\Theta}$ to be a {\it nearly sub-Gaussian random field w.r.t. $d$} if there is some $C \geq 1$ such that
\begin{equation}
\label{DefinitionNearlySubGaussian}
\ex\left[\exp\big(\lambda (X^{\theta}- X^{\vartheta})\big)\right]\leq C\cdot
\exp\big(\lambda^{2}~d(\theta,\vartheta)^{2}/2\big)\quad\mbox{for}~\theta, \vartheta\in\Theta\quad\mbox{and}\quad\lambda>0.
\end{equation}
Note that by symmetry, condition \eqref{DefinitionNearlySubGaussian} also holds for arbitrary $\lambda\in\R$.
In case of $C = 1$ this definition reduces to the ordinary notion of sub-Gaussian random fields.
For further information on sub-Gaussian random fields see e.g. \cite[Subsection 2.3]{GineNickl2016}. By a suitable change of the semimetric we may describe any nearly sub-Gaussian random field as a Gaussian random field.
\begin{lemma}
\label{Semimetrikwechsel}
If $(X^{\theta})_{\theta\in\Theta}$ is a nearly sub-Gaussian random field w.r.t. $d$, then it is a sub-Gaussian random field w.r.t. $\overline{d}\doteq \varepsilon\cdot d$ for some $\varepsilon > 1$.
\end{lemma}
\begin{proof}
Let $C > 1$ such that $(X^{\theta})_{\theta\in\Theta}$ satisfies \eqref{DefinitionNearlySubGaussian}. Then $\varepsilon\doteq \sqrt{12(2 C + 1)}$ is as required (cf. \cite[Lemma 2.3.2]{GineNickl2016}).
\end{proof}
The following properties of sub-Gaussian random fields are fundamental.
\begin{proposition}
\label{continuity sub Gaussian}
Let $X~\dot=~(X^{\theta})_{\theta\in\Theta}$ be a nearly sub-Gaussian random field  on some probability space $\tOFP$ w.r.t. $d$. 
If ${\cal D}(\Delta,d) < \infty$, then $X$ admits a separable version, and each separable version of $X$ has $\overline{\pr}$-almost surely bounded and $d$-uniformly continuous paths. In particular, 
for any separable version $\widehat{X}$ and for every $\overline{\theta}\in\Theta$, there is some random variable $U^{\overline{\theta}}$ on $\tOFP$ such that 
$$
\sup_{\theta\in\Theta}\widehat{X}^{\theta}\leq U^{\overline{\theta}} + \widehat{X}^{\overline{\theta}}\quad\overline{\pr}\mbox{-a.s.}\quad\mbox{and}\quad\ex_{\overline{\pr}}[\exp(p U^{\overline{\theta}})] < \infty\quad\mbox{for every}~p\in ]0,\infty[.
$$
\end{proposition}
\begin{proof}
In view of Lemma \ref{Semimetrikwechsel} we may assume without loss of generality that $X$ is a sub-Gaussian random field w.r.t. $d$.
It is already known that $X$ admits a separable version, and that each such version has $\overline{\pr}$-almost surely bounded and $d$-uniformly continuous paths (see \cite[Theorem 2.3.7]{GineNickl2016}). Now, let us fix any separable version $\widehat{X}$ of $X$ and an arbitrary $\overline{\theta}\in\Theta$. We have
$$
\sup_{\theta\in\Theta}\widehat{X}^{\theta}\leq \sup_{\theta\in\Theta}|\widehat{X}^{\theta} - \widehat{X}^{\overline{\theta}}| + \widehat{X}^{\overline{\theta}},
$$
and the process $\big(|\widehat{X}^{\theta} - \widehat{X}^{\overline{\theta}}|\big)_{\theta\in\Theta}$ is separable due to the separability of $\widehat{X}$. Then we may find some at most countable subset $\Theta_{0}\subseteq\Theta$ such that $\sup_{\theta\in\Theta}|\widehat{X}^{\theta} - \widehat{X}^{\overline{\theta}}| = \sup_{\theta\in\Theta_{0}}|\widehat{X}^{\theta} - \widehat{X}^ {\overline{\theta}}|$ $\overline{\pr}$-a.s.. Hence $U^{\overline{\theta}} := \sup_{\theta\in\Theta_{0}}|X^{\theta} - X^{\overline{\theta}}|$ defines a random variable on $\tOFP$ satisfying
\begin{equation*}
\label{domination}
U^{\overline{\theta}} = \sup_{\theta\in\Theta}|\widehat{X}^{\theta} - \widehat{X}^{\overline{\theta}}|~\overline{\pr}\mbox{-a.s.}.
\end{equation*}
It remains to show that $\ex_{\overline{\pr}}[\exp(p U^{\overline{\theta}})] < \infty$ holds for $p\in ]0,\infty[$. So let us fix $p\in ]0,\infty[$. Firstly, observe that $\widehat{X}$ is again a sub-Gaussian random field. Therefore by \cite[Lemma 2.3.1]{GineNickl2016} we have 
$$
\ex_{\overline{\pr}}\left[\exp\left(\left(\frac{X^{\theta}-X^{\vartheta}}{\sqrt{6}~ d(\theta,\vartheta)}\right)^{2}\right)\right]\leq 2\quad\mbox{for}~\theta, \vartheta\in\Theta,~d(\theta,\vartheta)\not= 0.
$$
Hence we may apply results from \cite{ViensVizcarra2007} w.r.t. the totally bounded semimetric $\overline{d}~\dot=~\sqrt{6}d$. Note that $(\widehat{X}^{\theta})_{\theta\in\Theta}$ is also separable w.r.t. $\overline{d}$, and $\overline{\Delta} = \sqrt{6}\Delta$ holds for the diameter $\overline{\Delta}$ w.r.t. $\overline{d}$. Since $N(\Theta,\overline{d};\varepsilon)\leq 
N(\Theta,d;\varepsilon/\sqrt{6})$ holds for every $\varepsilon > 0$, we obtain 
\begin{equation*}
\label{andere semimetrik}
\int_{0}^{\delta}\sqrt{\ln(N(\Theta,\overline{d};\varepsilon))}~d\varepsilon 
\leq
\int_{0}^{\delta}\sqrt{\ln(N(\Theta,d;\varepsilon/\sqrt{6}))}~d\varepsilon
= 
\sqrt{6} {\cal D}(\delta/\sqrt{6},d)\quad\mbox{for every}~\delta > 0.
\end{equation*} 
Then in view of \cite[Corollary 3.2]{ViensVizcarra2007} we may find some constant $C > 0$ such that
\begin{equation*}
\label{ViensVizcarra1}
\overline{\pr}\left(
\left\{U^{\overline{\theta}} > xC\sqrt{6}{\cal D}(\Delta,d)\right\}\right)\leq 2\exp\left(-\frac{x^{2}}{2}\right)\quad\mbox{for}~x\geq 1.
\end{equation*}
Furthermore
setting $\widehat{C}~\dot=~C\sqrt{6}{\cal D}(\Delta,d)$, we may observe 
\begin{align*}
\int_{1}^{\infty}\hspace*{-0.2cm}\overline{\pr}\left(
\left\{U^{\overline{\theta}} > x\widehat{C}\right\}\right)\exp(xp\widehat{C})~dx
\leq 
\int_{1}^{\infty}\hspace*{-0.2cm}2\exp\left(-\frac{x^{2}}{2}\right)\exp(xp\widehat{C})~dx
&\leq 2\sqrt{2\pi}\exp(p^{2}\widehat{C}^{2}/2).
\end{align*}
Then applying change of variable formula several times, we obtain
\begin{align*}
\int_{\exp(p\widehat{C})}^{\infty}\overline{\pr}\left(
\left\{\exp\big(pU^{\overline{\theta}}\big) > y\right\}\right)~dy
&=
p\widehat{C}\int_{1}^{\infty}\overline{\pr}\left(
\left\{U^{\overline{\theta}} > \widehat{C}u\right\}\right)\exp(p\widehat{C}u)~du 
< 
\infty.
\end{align*}
Hence
$$
\ex_{\overline{\pr}}\big[\exp\big(pU^{\overline{\theta}}\big)\big] 
=
\int_{0}^{\infty}\overline{\pr}\left(
\left\{\exp\big(pU^{\overline{\theta}}\big) > y\right\}\right)~dy 
< \infty
$$
which completes the proof.
\end{proof}
\section*{Acknowledgements}
The authors would like to thank Mikhail Urusov for fruitful discussions and helpful remarks.

\end{document}